\documentclass[leqno]{amsart}
\usepackage{times}
\usepackage{amsfonts,amssymb,amsmath,amsgen,amsthm}
\usepackage{hyperref}

\theoremstyle{plain}
\newtheorem{theorem}{Theorem}[section]

\newtheorem{lemma}[theorem]{Lemma}
\newtheorem{corollary}[theorem]{Corollary}
\newtheorem{proposition}[theorem]{Proposition}
\newtheorem{hyp}[theorem]{Assumption}
\theoremstyle{remark}
\newtheorem{remark}[theorem]{Remark}
\newtheorem*{notation}{Notation}

\def\dis
{\displaystyle}

\def\R{{\mathbf R}}
\def\N{{\mathbf N}}
\def\Z{{\mathbf Z}}
\def\O{\mathcal O}

\def\({\left(}
\def\){\right)}
\def\<{\left\langle}
\def\>{\right\rangle}
\def\le{\leqslant}
\def\ge{\geqslant}

\def\Tend#1#2{\mathop{\longrightarrow}\limits_{#1\rightarrow#2}}

\def\d{{\partial}}
\def\eps{\varepsilon}

\def\apsi{\psi_{\rm app}^\eps}

\DeclareMathOperator{\RE}{Re}
\DeclareMathOperator{\IM}{Im}

\numberwithin{equation}{section}

\begin{document}

\title[Interaction of coherent states for Hartree
equations]{Interaction of coherent states for Hartree
equations}  
\author[R. Carles]{R\'emi Carles}
\address{CNRS \& Univ. Montpellier~2\\Math\'ematiques
\\CC~051\\34095 Montpellier\\ France}
\email{Remi.Carles@math.cnrs.fr}

\begin{abstract}
 We consider the Hartree equation with a smooth kernel and an external
 potential, in the
 semiclassical regime. We analyze the 
 propagation of two initial wave packets, and show different possible
 effects of the interaction, according to the size of the
 nonlinearity in terms of the semiclassical parameter. We show three
 different sorts of nonlinear phenomena. In
 each case, the structure 
 of the wave as a sum of two coherent states is preserved. However, the
 envelope and the center (in phase space) of these two wave packets
 are affected by nonlinear interferences, which are described precisely. 
\end{abstract}
\thanks{This work was supported by the French ANR project
  R.A.S. (ANR-08-JCJC-0124-01).}
\maketitle

\setcounter{tocdepth}{1}
\tableofcontents

\section{Introduction}
\label{sec:intro}
Consider the following Hartree equation in the semiclassical regime
$\eps\to 0$:
 \begin{equation}
  \label{eq:r3alpha}
  i\eps \d_t \psi^\eps + \frac{\eps^2}{2}\Delta \psi^\eps =
  V(t,x)\psi^\eps +\eps^\alpha\(K\ast 
  |\psi^\eps|^2\) \psi^\eps,\quad t\in
    \R_+=[0,\infty),\ x\in \R^d, 
\end{equation}
where $\alpha\ge 0$, $K:\R^d\to \R$, $V:\R_{+}\times \R^d\to \R$,
$d\ge 1$.
Equation~\eqref{eq:r3alpha} appears for instance as a model to study
superfluids, with application to Bose--Einstein condensation: 
in \cite{Be99,BeRo99}, the kernel $K$ is given by the formula
\begin{equation*}
  K(x) = \(a_1 + a_2 |x|^2 +a_3 |x|^4\) e^{-A^2|x|^2}+a_4e^{-B^2|x|^2} ,\quad
  a_1,a_2,a_3,a_4, A,B\in \R. 
\end{equation*}
Assume
\begin{equation}
  \label{eq:ci}
  \psi^\eps(0,x)=\eps^{-d/4}
  a\left(\frac{x-q_0}{\sqrt\eps}\right)
e^{i{(x-q_0)\cdot p_0/\eps}},\quad a\in{\mathcal S}({\R}^d),\quad
q_0,p_0\in \R^d.
\end{equation}
Such initial data are called semiclassical wave packets, or coherent
states. They correspond to a wave function which is equally localized
in space and in frequency (at scale $\sqrt\eps$), so the uncertainty
principle is optimized in terms of $\eps$: 
the three quantities
\begin{equation*}
  \|\psi^\eps(0)\|_{L^2(\R^d)},\quad \left\|
\(\sqrt\eps \nabla-i\frac{p_0}{\sqrt\eps}\)
    \psi^\eps(0)\right\|_{L^2(\R^d)},\quad \text{and } 
\quad \left\|\frac{x-q_0}{\sqrt\eps}
  \psi^\eps(0)\right\|_{L^2(\R^d)}
\end{equation*}
have the same order of magnitude, $\O(1)$, as $\eps\to 0$. In the
linear case $K=0$, another reason why such
specific initial data are particularly interesting is that the
superposition principle is available: if we can describe $\psi^\eps$
in the case \eqref{eq:ci}, then the evolution of a sum of initial wave
packets of the form \eqref{eq:ci} is simply the sum of the evolutions
of each initial wave packet. In this paper, we address this question
in the nonlinear setting. We describe several nonlinear
interference phenomena in the case where $K\not =0$ is smooth, and
$\psi^\eps(0,x)$ is the sum of two such wave packets. 

The value of the parameter $\alpha$ in \eqref{eq:r3alpha} measures
the strength of the nonlinear interaction in the limit $\eps\to 0$. 
In \cite{CaFe11}, where the Hartree
nonlinearity is replaced by a local nonlinearity, it is established
that if nonlinear effects are critical in terms of semiclassical
dynamics (that is, the value of $\alpha$ is critical, see
\S\ref{sec:notion} for this notion), then despite the 
fact that the
problem is nonlinear, the superposition principle remains valid,
in the limit $\eps\to 0$. In \cite{CaCa11}, the case of
a homogeneous Hartree
nonlinearity $K(x)=\lambda |x|^{-\gamma}$ is considered: conclusions
similar to those in \cite{CaFe11} are proven.  
In these two frameworks, the description of
the wave packet dynamics in a ``supercritical'' case (nonlinear
effects are stronger than in the critical case) is an open
question, even on a formal level. On the other hand, in the case of a
\emph{smooth} Hartree kernel, the 
propagation of a single wave packet has been described in
supercritical regimes (\cite{APPP11,CaCa11}).  In this paper, we prove
that in the critical regime, nonlinear interferences affect the
propagation of two initial wave packets at
leading order, in contrast with the case of a
homogeneous kernel. We 
also describe the nonlinear interactions in supercritical regimes,
where even stronger interferences are present. In
all cases, we prove a convergence result on all finite time
intervals ($t\in [0,T]$ with $T$ independent of $\eps$), as $\eps\to 0$.

 \begin{hyp}\label{hyp:gen}
   The external potential $V$ is $C^3$, real-valued, and at most quadratic
  in space:
  \begin{equation*}
 V\in C^3(\R_{+}\times \R^d;\R),\quad \text{and}\quad   \d_x^\beta V\in
 L^\infty\(\R_{+}\times\R^d\),\quad  |\beta|=  2,3.
  \end{equation*}
The kernel $K$ is $C^3$, real-valued, bounded as well as  its first
three derivatives: 
  \begin{equation*}
  K\in C^3\cap W^{3,\infty}(\R^d;\R).
  \end{equation*}
\end{hyp}
Consider the Hamiltonian flow:
\begin{equation}\label{eq:traj}
 \dot q(t)=p(t),\;\;\dot p(t)=-\nabla V\(t,q(t)\);\quad
 q(0)=q_0,\;p(0)= p_0.
\end{equation}
The regularity of $V$ implies that \eqref{eq:traj} has a unique,
global solution 
\begin{equation*}
  t\mapsto \(q(t),p(t)\)\in
C^3\(\R_+;\R^{2d}\).
\end{equation*}
 Since we shall consider only bounded time
intervals in this paper, the growth in time of the classical
trajectories is not discussed.

\subsection{The linear case $K=0$}\label{sec:linear} 
Introduce the function
\begin{equation*}\label{eq:linearvarphi}
\varphi_{{\rm
lin}}^{\eps}(t,x)=\eps^{-d/4}u^{\rm lin}\(t,
\frac{x-q(t)}{\sqrt{\eps}}\)e^{i(S(t) 
+p(t)\cdot(x-q(t)))/\eps},
\end{equation*}
where $(q,p)$ is given by \eqref{eq:traj}, the classical action is
given by 
\begin{equation}\label{eq:classicalaction}
S(t)=\int_0^t  \(\frac{1}{2} |p(s)|^2-V(s,q(s))\)ds,
\end{equation}
and the envelope $u^{\rm lin}=u^{\rm lin}(t,y) $ solves
\begin{equation}\label{eq:linearu}
 i\partial_{t}u^{\rm lin}+\frac{1}{2}\Delta u^{\rm
   lin}=\frac{1}{2}\<y, \nabla^2 V\(t,q(t)\)y\>u^{\rm lin}\quad ;\quad 
u^{\rm lin}(0,y)=a(y),
\end{equation}
where the notation $\nabla^2$ stands for the Hessian matrix,
and since the space variable for $u^{\rm lin}$ is
  $y$, $\Delta$ stands for $\Delta_y$.
The following lemma is standard, see e.g.
\cite{BGP99,CoRo97,CoRo06,CoRo07,Hag80,HaJo00,HaJo01} and references
therein.
\begin{lemma}\label{lemlinear}
Let $a\in \mathcal{S}(\R^{d})$, and $\psi^\eps$ solve \eqref{eq:r3alpha}
with $K=0$, and \eqref{eq:ci}. There exist positive
constants $C$ 
and $C_{1}$ independent of $\eps$
 such that
 \begin{equation*}
 \|\psi^{\eps}(t)-\varphi_{{\rm
     lin}}^{\eps}(t)\|_{L^{2}(\R^{d})}\le 
 C\sqrt{\eps} e^{C_{1}t},\quad \forall t\ge 0.
 \end{equation*}
 In particular, there exists  $c>0$ independent of $\eps$ such that
 \begin{equation*}
 \sup_{0\le t\le c\ln\frac{1}{\eps}}\|\psi^{\eps}(t)
 -\varphi_{{\rm lin}}^{\eps}(t)\|_{L^{2}(\R^{d})}\Tend
 \eps 0 0. 
 \end{equation*}
 \end{lemma}

\subsection{Nonlinear case: notion of criticality}
\label{sec:notion}
In the nonlinear case $K\not =0$, the following distinction was
established in \cite{CaCa11}: 
\begin{itemize}
\item If $\alpha>1$, nonlinear effects are negligible at leading
  order: with the same function $\varphi_{\rm lin}^\eps$ as in the previous
  section, there exists $C>0$ such that
 \begin{equation*}
 \|\psi^{\eps}(t)-\varphi_{{\rm
     lin}}^{\eps}(t)\|_{L^{2}(\R^{d})}\le 
 C\sqrt{\eps} e^{C t},\quad \forall t\ge 0.
 \end{equation*}
\item If $\alpha=1$, nonlinear effects become relevant at leading
  order (unless $K(0)=0$): there exists $C>0$ such that
 \begin{equation*}
 \left\|\psi^{\eps}(t)-\varphi_{{\rm
     lin}}^{\eps}(t)e^{-itK(0)\|a\|_{L^2}^2}\right\|_{L^{2}(\R^{d})}\le 
 C\sqrt{\eps} e^{C t},\quad \forall t\ge 0.
 \end{equation*}
\end{itemize}
From this point of view, the case $\alpha=1$ is critical: the
supercritical behavior is described in two cases, $\alpha=1/2$
(\cite{CaCa11}) and
$\alpha=0$ (\cite{APPP11,CaCa11}). The approximate solution derived in
these two cases may be viewed as a particular case of the approximate
solution presented below,  when one of the two initial
wave packets is zero, so we choose not
to be more explicit about these two cases here. Other cases could be
described as well: the 
case $\alpha\in (0,1/2)$ is similar to the case $\alpha=0$, and the
case $\alpha\in (1/2,1)$ is similar to the case $\alpha=1/2$, up to
several modifications in the notations essentially. 
\smallbreak

  In the case $\alpha>1$, nonlinear effects are negligible at leading
  order, so the superposition principle remains: the
  nonlinear evolution of two (or more) initial wave packets is well
  approximated by the sum of the linear evolutions of each wave
  packet. We will see that when $\alpha\le 1$, nonlinear interferences
  affect the behavior of $\psi^\eps$ at leading order. 
\smallbreak

Throughout this paper, for $k\in \N$, we will denote by
\begin{equation*}
  \Sigma^k = \left\{ f\in L^2(\R^d)\ ;\ \| f \| _{\Sigma^k}:=
    \sum_{|\alpha|+|\beta|\le  k}\left\lVert x^\alpha \d_x^\beta
      f\right\rVert_{L^2(\R^d)}<\infty\right\}, 
\end{equation*}
and $\Sigma^1=\Sigma$. 
As established in \cite{CaCa11}, if
$\psi^\eps(0,\cdot)\in L^2(\R^d)$, then under
Assumption~\ref{hyp:gen}, \eqref{eq:r3alpha} has a unique solution
$\psi^\eps \in C(\R_+;L^2(\R^d))$, regardless of the value of
$\alpha$\footnote{To be complete, the regularity assumption on $V$ in
\cite{CaCa11} is stronger, but Assumption~\ref{hyp:gen} is enough.},
and 
\begin{equation*}
  \|\psi^\eps(t)\|_{L^2(\R^d)} = \|\psi^\eps(0)\|_{L^2(\R^d)}, \quad
  \forall t\ge 0. 
\end{equation*}
\subsection{Critical case: $\alpha=1$}
\label{sec:crit}
We now consider \eqref{eq:r3alpha} in the case of two initial wave
packets: \eqref{eq:ci} is replaced by
\begin{equation}\label{eq:ci2}
\psi^{\eps}(0,x)= \eps^{-d/4} \sum_{j=1,2} a_j
      \(\frac{x-q_{j0}} {\sqrt{\eps}}\) e^{i(x-q_{j0})\cdot
        p_{j0}/ \eps},
\end{equation}
with $(q_{10},p_{10})\not = (q_{20},p_{20})$. Let $(q_j,p_j)$ be the
solution to \eqref{eq:traj} with initial data $(q_{j0},p_{j0})$, and
$S_j$ the associated classical action given by
\eqref{eq:classicalaction}. Define the approximate solution as 
\begin{equation}
  \label{eq:approxdef}
  \psi_{\rm app}^\eps(t,x)=\eps^{-d/4} \sum_{j=1,2}u_j 
\left(t,\frac{x-q_j(t)}{\sqrt\eps}\right)e^{i\left(S_j(t)+p_j(t)\cdot
    (x-q_j(t))\right)/\eps},
\end{equation}
where the envelopes $u_j$ are given by the  formulas:
\begin{equation}
  \label{eq:envcrit}
\left\{
\begin{aligned}
  u_1(t,y)&=u^{\rm lin}_1(t,y)e^{-i t
    K(0)\|a_1\|_{L^2}^2-i\|a_2\|_{L^2}^2 \int_0^t K\(q_1(s)-q_2(s)\)ds},\\
u_2(t,y)&=u^{\rm lin}_2(t,y)e^{-i t
    K(0)\|a_2\|_{L^2}^2-i\|a_1\|_{L^2}^2 \int_0^t K\(
    q_2(s)-q_1(s)\)ds},
\end{aligned}
\right.
\end{equation}
with obvious notations adapted from  \eqref{eq:linearu}.
\begin{theorem}\label{theo:critical}
  Let $d\ge 1$, $V,K$ satisfying Assumption~\ref{hyp:gen}. Let
  $a_1,a_2\in \Sigma^3$, and $\psi^\eps$ be the solution to
  \eqref{eq:r3alpha} with $\alpha=1$ and initial data
  \eqref{eq:ci2}. Then for any $T>0$ independent of $\eps$, there
  exists $C>0$ independent of $\eps$ such that
  \begin{equation*}
    \sup_{t\in [0,T]}\left\| \psi^\eps(t)-\apsi
      (t)\right\|_{L^2(\R^d)}\le C\sqrt\eps,
  \end{equation*}
where $\apsi$ is given by \eqref{eq:approxdef}--\eqref{eq:envcrit}. 
\end{theorem}
The nonlinear effects are described at leading order by the
exponentials in \eqref{eq:envcrit}. Even in the case of a single
initial wave packet (say $a_2=0$), the nonlinearity affects the
envelope by a phase self-modulation. The second terms in the
exponentials describe the effect of nonlinear interferences,
which are not a simple superposition in general. 
\smallbreak

  As pointed out above, it may be surprising to notice that even in
  the critical case 
$\alpha=1$, nonlinear interferences are present at leading order. This
is in sharp contrast with the case of an homogeneous kernel,
$K(x)=\lambda |x|^{-\gamma}$, $0<\gamma<\min (2,d)$. It was shown in
\cite{CaCa11} that in this case, the critical value
for $\alpha$ is 
$\alpha_c =1+\gamma/2$, and that when $\alpha=\alpha_c$, the
superposition principle remains, even though the nonlinearity affects
the propagation of a single wave packet at leading order (the envelope
equation is nonlinear). 

\subsection{Case  $\alpha=1/2$}
\label{sec:super12}
The approximate solution is now constructed as follows. 
The pairs $(q_j,p_j)$, $j=1,2$, are still given by the usual classical
flow \eqref{eq:traj}. On the other hand, we modify the actions, and
make them $\eps$-dependent:
 \begin{equation}
  \label{eq:actionhalf}
\left\{
  \begin{aligned}
S^\eps_1(t) &= \int_{0}^{t}\(\frac{1}{2}
    |p_1(s)|^2-V(s,q_1(s))\)ds\\
&\phantom{=}-t\sqrt{\eps} K(0)\|a_1\|_{L^2(\R^d)}^2
    -\sqrt\eps \|a_2\|_{L^2(\R^d)}^2\int_0^t K\(q_1(s)-q_2(s)\)ds ,\\
S^\eps_2(t) &= \int_{0}^{t}\(\frac{1}{2}
    |p_2(s)|^2-V(s,q_2(s))\)ds\\
&\phantom{=}-t\sqrt{\eps} K(0)\|a_2\|_{L^2(\R^d)}^2
    -\sqrt\eps \|a_1\|_{L^2(\R^d)}^2\int_0^t K\(q_2(s)-q_1(s)\)ds.
\end{aligned}
\right.
\end{equation}
Consider the system of Schr\"odinger equations
\begin{equation}
  \label{eq:envhalf}
  \left\{
\begin{aligned}
i\d_t \tilde u_1 +\frac{1}{2}\Delta \tilde u_1 & =
\frac{1}{2}\<y,\nabla^2 V\(t,q_1(t)\)y\>\tilde u_1+\|a_1\|_{L^2}^2
y\cdot \nabla K(0) \tilde u_1\\
&\phantom{=}+ \|a_2\|_{L^2}^2
y\cdot \nabla K\(q_1(t)-q_2(t)\) \tilde u_1,\\
i\d_t \tilde u_2 +\frac{1}{2}\Delta \tilde u_2 &=
\frac{1}{2}\<y,\nabla^2 V\(t,q_2(t)\)y\>\tilde u_2+ \|a_2\|_{L^2}^2
y\cdot \nabla K(0) \tilde u_2\\
&\phantom{=} + \|a_1\|_{L^2}^2
y\cdot \nabla K\(q_2(t)-q_1(t)\) \tilde u_2,
\end{aligned}
\right.
\end{equation}
with initial data $a_1$ and $a_2$, respectively.
These are two linear equations with time dependent potentials, which
are polynomial in $y$, of degree (at most) two. The following result
is classical, see e.g. \cite{ReedSimon2,Fujiwara79,Fujiwara}:
\begin{lemma}\label{lem:exist12}
  For $j=1,2$, let $a_j\in L^2(\R^d)$, and $(q_j,p_j)\in C^3(\R_+;\R^{2d})$
  given by \eqref{eq:traj}. There exists a unique solution
  $(\tilde u_1,\tilde u_2)\in C(\R_+;L^2(\R^d))^2$ to
  \eqref{eq:envhalf} such that 
  $(\tilde u_1,\tilde u_2)_{\mid t=0} = (a_1,a_2)$. In addition, 
  the following  conservations hold:
  \begin{equation*}
    \|\tilde u_j(t)\|_{L^2(\R^d)}= \|a_j\|_{L^2(\R^d)},\quad \forall t\ge 0,\
    j=1,2. 
  \end{equation*}
\end{lemma}
To define the envelopes in \eqref{eq:approxdef}, set
\begin{equation}\label{eq:solenvhalf}
\begin{aligned}
&  u_1(t,y_1) = \tilde u_1(t,y_1)\exp\(i\int_0^t\(\nabla K(0)\cdot
  \tilde G_1(s) + \nabla K\(q_1(s)-q_2(s)\)\cdot \tilde G_2(s)\)ds \),\\
& u_2(t,y_2) = \tilde u_2(t,y_2)\exp\(i\int_0^t\(\nabla K(0)\cdot
  \tilde G_2(s) + \nabla K\(q_2(s)-q_1(s)\)\cdot \tilde G_1(s)\)ds \),
\end{aligned}
\end{equation}
where $\dis  \tilde G_j(t) = \int_{\R^d}z|\tilde u_j(t,z)|^2dz.$ Since
$\tilde G_j$ is a nonlinear function of $\tilde u_j$, the system
formed by $(u_1,u_2)$ is nonlinear, with a nonlinear coupling:
nonlinear interferences are present both in rapid oscillations --- the
modified actions generate $\sqrt\eps$-oscillations in time ---
and in the envelopes. The presence of the functions 
$\tilde G_j$ in the above formulas reveals non-local (in space)
nonlinear phenomena concerning the envelopes in $\apsi$. Since the
problem is now supercritical, it should not be surprising that
stronger regularity properties are assumed in the following result
(see Remark~\ref{rem:sur}).
\begin{theorem}\label{theo:alpha12}
  Let $d\ge 1$. 
 Assume 
   that $V$ and $K$ are real-valued and satisfy:
 \begin{align*}
 & V\in C^5(\R_{+}\times \R^d;\R),\quad \text{and}\quad   \d_x^\beta V\in
 L^\infty\(\R_{+}\times\R^d\),\quad  2\le |\beta|\le 5.\\
 &K\in W^{5,\infty}(\R^d;\R).
 \end{align*}
Let  $a_1,a_2\in \Sigma^5$, and $\apsi$ be given by
  \eqref{eq:approxdef}--\eqref{eq:actionhalf}--\eqref{eq:envhalf}--\eqref{eq:solenvhalf}. 
Then for any $T>0$ independent of $\eps$, there
  exists $C>0$ independent of $\eps$ such that
  \begin{equation*}
    \sup_{t\in [0,T]}\left\| \psi^\eps(t)-\apsi
      (t)\right\|_{L^2(\R^d)}\le C\sqrt\eps.
  \end{equation*}
\end{theorem}

\subsection{Case $\alpha=0$}
\label{sec:super0}
In this last case, nonlinear interferences affect
even the geometric properties of the wave packets, in contrast with
the cases $\alpha=1$ and $\alpha=1/2$. The trajectories
are required to evolve according to the system
\begin{equation}\label{eq:trajmodif}
  \left\{
\begin{aligned}
 & \dot q_1(t) = p_1(t),\\
& \dot p_1(t) = -\nabla V\(t,q_1(t)\) -
  \|a_1\|_{L^2}^2 \nabla K(0)- \|a_2\|_{L^2}^2 \nabla
  K\(q_1(t)-q_2(t)\),\\
&\dot q_2(t) = p_2(t),\\
& \dot p_2(t) = -\nabla V\(t,q_2(t)\) -
  \|a_2\|_{L^2}^2 \nabla K(0)- \|a_1\|_{L^2}^2 \nabla
  K\(q_2(t)-q_1(t)\).
\end{aligned}
\right.
\end{equation}
Unless $\nabla K$ is a constant (which would implies that $K$ is
constant, a trivial case), one cannot decouple the
unknowns $(q_1,p_1)$ and $(q_2,p_2)$: the coupling cannot by undone,
and the ``good unknown'' is $(q_1,p_1,q_2,p_2)\in \R^{4d}$.
In view of Assumption~\ref{hyp:gen}, Cauchy--Lipschitz Theorem implies:
\begin{lemma}\label{lem:cauchylip}
  For $j=1,2$, let $(q_{j0},p_{j0})\in \R^{2d}$. If $V$ and $K$ satisfy
  Assumption~\ref{hyp:gen}, then \eqref{eq:trajmodif} has a unique
  solution $(q_1,p_1,q_2,p_2)\in C^3(\R_+;\R^{4d})$. 
\end{lemma}
\begin{remark}[Hamiltonian structure]\label{rem:hamil}
  If the external potential $V$ does not depend on time, $\d_t V=0$,
  and the kernel $K$ is even, $K(-x)=K(x)$ for all $x\in\R^d$, then
  the Hartree equation \eqref{eq:r3alpha} has a Hamiltonian
  structure. In the case $\alpha=0$, the following energy is
  independent of $t$,
  \begin{equation*}
    \frac{\eps^2}{2}\|\nabla\psi^\eps(t)\|_{L^2}^2+
\int_{\R^d}V(x)|\psi^\eps(t,x)|^2 
    dx +
    \frac{1}{2}\iint_{\R^{2d}}K(x-y)|\psi^\eps(t,y)|^2|\psi^\eps(t,x)|^2dxdy.
  \end{equation*}
Note that since $K$ is even, $\nabla K(0)=0$, and $\nabla K(q_2-q_1)=-
\nabla K(q_1-q_2) $.
In that case, the system of modified trajectories \eqref{eq:trajmodif}
is also Hamiltonian, as can be seen from the approach presented in
\cite{GHM08}. Given the state variable $z=
(q_1,p_1,q_2,p_2)^T$, let
\begin{align*}
  H(t,z)= \alpha_1\(\frac{1}{2}|p_1|^2+V(q_1)\) +
  \alpha_2\(\frac{1}{2} |p_2|^2+V(q_2) \)+
 \alpha_1\alpha_2 K\(q_1-q_2\), 
\end{align*}
where $\alpha_j=\|a_j\|_{L^2}^2$. The system \eqref{eq:trajmodif}
has the Hamiltonian structure 
\begin{equation*}
  \frac{dz}{dt} = JD_z H(t,z)\quad\text{with}\quad J=
  \begin{pmatrix}
    0 & 1/\alpha_1 & 0 & 0 \\
-1/\alpha_1 & 0 & 0 & 0\\
0 & 0&0 &1/\alpha_2\\
0 & 0&-1/\alpha_2 &0
  \end{pmatrix}.
\end{equation*}
One checks indeed that the function $H$ is conserved along solutions
of \eqref{eq:trajmodif}.
\end{remark}
Before defining the modified actions, we have to construct also the
envelopes. Consider the coupled, nonlinear system
\begin{equation}
  \label{eq:systenv0}
  \left\{
\begin{aligned}
  i\d_t & u_1 +\frac{1}{2}\Delta u_1= \frac{1}{2}\<y, M_1(t)y\>u_1 -
\<\nabla^2K(0) G_1(t),y\>u_1  \\
& -
\<\nabla^2K(q_1-q_2) G_2(t),y\>u_1 +\frac{1}{2} \(\int
\<z,\nabla^2K(0)z\>|u_1(t,z)|^2dz\) u_1\\
&+\frac{1}{2}\( \int
\<z,\nabla^2K(q_1-q_2)z\>|u_2(t,z)|^2dz\) u_1, \\
i\d_t & u_2 +\frac{1}{2}\Delta u_2= \frac{1}{2}\<y, M_2(t)y\>u_2 -
\<\nabla^2K(0) G_2(t),y\>u_2  \\
&-
\<\nabla^2K(q_2-q_1) G_1(t),y\>u_2 +\frac{1}{2} \(\int
\<z,\nabla^2K(0)z\>|u_2(t,z)|^2dz\) u_2\\
&+\frac{1}{2} \(\int
\<z,\nabla^2K( q_2-q_1)z\>|u_1(t,z)|^2dz\) u_2, 
\end{aligned}
\right.
\end{equation}
where the functions $q_j$ are assessed at time $t$, and we have denoted 
\begin{align}
  G_j(t)&=\int_{\R^d}z|u_j(t,z)|^2dz,\quad j=1,2,\label{def:Gj}\\
M_1(t)&=\|a_1\|_{L^{2}(\R^{d})}^{2}\nabla^2
K(0)+\|a_2\|_{L^{2}(\R^{d})}^{2}\nabla^2 K\(q_1(t)-q_2(t)\)\label{def:M1}\\
&\quad+
\nabla^2_xV\(t,q_1(t)\),\notag\\
M_2(t)&=\|a_2\|_{L^{2}(\R^{d})}^{2}\nabla^2
K(0)+\|a_1\|_{L^{2}(\R^{d})}^{2}\nabla^2 K\( q_2(t)-q_1(t)\)\label{def:M2}\\
&\quad+
\nabla^2_xV\(t,q_2(t)\).\notag
\end{align}
The system defining the envelopes is more nonlinear than the cases
$\alpha=1$ and $\alpha=1/2$, and, as in the case $\alpha=1/2$,
involves nonlinear terms which are non-local in space.
In Section~\ref{sec:refined}, we prove the following result:
\begin{proposition}\label{prop:existenv0}
  Let $(q_1,p_1,q_2,p_2)$ be given by Lemma~\ref{lem:cauchylip}, and
  $a_1,a_2\in \Sigma^k$ with $k\ge 1$. Then \eqref{eq:systenv0} has a
  unique solution $(u_1,u_2)\in C(\R_+;\Sigma^k)$ with initial data
  $(a_1,a_2)$. In addition, 
  the following conservations hold:
  \begin{equation*}
    \|u_j(t)\|_{L^2(\R^d)}= \|a_j\|_{L^2(\R^d)},\quad \forall t\ge 0,\
    j=1,2. 
  \end{equation*}
\end{proposition}
We can then define the modified, $\eps$-dependent actions:
\begin{align*}
  S_1^\eps(t) &= \int_0^t \Big(\frac{1}{2}|p_1(s)|^2 - V\(s,q_1(s)\)
 - K(0)\|a_1\|_{L^2}^2- K\(q_1(s)-q_2(s)\)\|a_2\|_{L^2}^2 \\
&\phantom{= \int_0^t \Big(}+\sqrt \eps
  \nabla K(0)\cdot G_1(s)+\sqrt\eps \nabla K\(q_1(s)-q_2(s)\)\cdot
  G_2(s)\Big)ds,\\
     S_2^\eps(t) &= \int_0^t \Big(\frac{1}{2}|p_2(s)|^2 - V\(s,q_2(s)\)
 - K(0)\|a_2\|_{L^2}^2- K\(q_2(s)-q_1(s)\)\|a_1\|_{L^2}^2 \\
&\phantom{= \int_0^t \Big(}+\sqrt \eps
  \nabla K(0)\cdot G_2(s)+\sqrt\eps \nabla K\(q_2(s)-q_1(s)\)\cdot
  G_1(s)\Big)ds.
\end{align*}
\begin{theorem}\label{theo:alpha0}
  Let $d\ge 1$. Assume 
  that $V$ and $K$ are real-valued and satisfy:
 \begin{align*}
 & V\in C^7(\R_{+}\times \R^d;\R),\quad \text{and}\quad   \d_x^\beta V\in
 L^\infty\(\R_{+}\times\R^d\),\quad  2\le |\beta|\le 7.\\
 &K\in W^{7,\infty}(\R^d;\R). 
 \end{align*}
Let  $a_1,a_2\in \Sigma^7$. There
  exist $\theta_1,\theta_2\in C^2(\R_+;\R)$, with
 $   \theta_j(0)=\dot \theta_j(0)=0,$
such   that the following holds. 
For any $T>0$ independent of $\eps$, there
  exists $C>0$ independent of $\eps$ such that
  \begin{equation*}
  \sup_{t\in [0,T]}\left\| \psi^\eps(t)-
 \sum_{j=1,2}\varphi_j^\eps(t)
e^{i\theta_j(t)}
\right\|_{L^2(\R^d)}\le C\sqrt\eps,
  \end{equation*}
where we have denoted $\displaystyle \varphi_j^\eps(t,x) = \eps^{-d/4} u_j 
\left(t,\frac{x-q_j(t)}{\sqrt\eps}\right)e^{i\left(S_j^\eps(t)+p_j(t)\cdot
    (x-q_j(t))\right)/\eps}$. 
\end{theorem}
In general, the phase shifts $\theta_j$ are not identically zero. In
Section~\ref{sec:DA2}, we give the expression of these functions,
which is probably a bit too involved to present at this stage (see
Equation~\eqref{eq:theta}), and
check that in general, $(\ddot \theta_1(0),\ddot \theta_2(0))\not
=(0,0)$. Such modifications do not appear in the case of a single wave
packet studied in \cite{APPP11,CaCa11}. Technically, the reason is
two-fold. First, in  \cite{APPP11,CaCa11}, it is assumed that $\nabla
K(0)=0$, so the first line in \eqref{eq:theta} vanishes. Then, the
second line in \eqref{eq:theta} accounts for the presence of two wave
packets, and measures some coupling through a phase modulation: it
vanishes in the case of a single wave packet.  

\subsection{Comments}
\label{sec:comments}

\subsubsection*{The results}

In the three cases studied here, the
interferences are nonlinear.
They always affect the envelopes. In the case $\alpha=1/2$, they
affect moreover the action, and in the case $\alpha=0$,  the
system \eqref{eq:trajmodif} reveals a nonlinear coupling of
the trajectories, so all the terms involved in $\apsi$ are influenced
by the nonlinearity. Our results are original
even in the case $V=0$. 

Nonlinear interferences always carry a
non-local in time aspect. Even if $K$ is decaying at infinity, the
interactions ignore the mutual 
distance of the two wave packets: no matter how large $q_1(t)-q_2(t)$
is, nonlinear interferences affect the solution at order $\O(1)$ on
finite time intervals, as $\eps\to 0$.   

\smallbreak

Our results yield a unified picture concerning Wigner measures (see
e.g. \cite{BurqMesures,GMMP,LionsPaul}):
\begin{corollary}
  In all the cases $\alpha=1$, $\alpha=1/2$ and $\alpha=0$, and under
  the Assumptions of Theorem~\ref{theo:critical},
  Theorem~\ref{theo:alpha12} and Theorem~\ref{theo:alpha0},
  respectively, the Wigner measure of $\psi^\eps$ is given by
  \begin{equation*}
    w(t, x,
\xi)=\sum_{j=1,2}\|a_j\|_{L^{2}(\R^{d})}^2\delta\(x-q_j(t)\)\otimes
\delta \(\xi-p_j(t)\),
  \end{equation*}
with $(q_j,p_j)$ given by the standard Hamiltonian flow
\eqref{eq:traj} in the cases $\alpha=1$ and $\alpha=1/2$, and 
$(q_1,p_1,q_2,p_2)$ given by \eqref{eq:trajmodif} in the case $\alpha=0$. 
\end{corollary}
To be complete, the proof of this corollary relies also on the results
established in Section~\ref{sec:mixed}. 
In the two cases $\alpha=1$ and $\alpha=1/2$, the Wigner measure of
$\apsi$ is not affected by the 
nonlinearity, 
even though we have seen
that the Hartree nonlinearity 
does affect the leading order behavior of the wave function, and that
nonlinear exchanges are present at leading order. In the case
$\alpha=0$, nonlinear effects 
alter the Wigner measure, \emph{even when $\nabla K(0)=0$}, a case
which is 
often encountered in Physics (typically, $K(-x)=K(x)$, so the Hartree
nonlinearity 
has an Hamiltonian structure). In other words, the Wigner measure of
$\psi^\eps$ is always affected by nonlinear interferences. This is in
contrast with the case of a single initial wave packet, where the
trajectory $(q,p)$ is modified as if an electric field $\|a\|_{L^2}^2
\nabla K(0)\cdot x$ was added to the initial Hamiltonian
$-\frac{1}{2}\Delta +V$: if $\nabla K(0)=0$, then the Wigner measure
ignores nonlinear effect even in the case $\alpha=0$ (see
\cite{APPP11,CaCa11}).
\smallbreak

Note that if $\nabla
K(0)\not =0$ (a case which is not necessarily physically relevant,
from the above remark), Theorem~\ref{theo:alpha0} is new \emph{even in
  the case of a single wave packet}. 
\smallbreak

In this paper, we treat the case of two initial wave packets: our
approach can easily be generalized to any (finite) number of initial
wave packets, the main difference being that formulas get more and
more involved as the number of initial wave packets increases (but the
main analytical aspects are essentially the same). 
\smallbreak

We have examined the leading order behavior of
the exact solution, up to an error of order $\O(\sqrt\eps)$: like in
\cite{APPP11}, $\psi^\eps$ could be approximated by a series involving
powers of $\sqrt\eps$, up to arbitrary order. This statement is
made more precise in 
\S\ref{sec:DA2} (see in particular Equation~\eqref{eq:DAgeneral}):
to prove Theorem~\ref{theo:alpha0}, the 
asymptotic expansion of the main unknown functions has to be pushed
one step further than in the cases $\alpha=1$ and $\alpha=1/2$. 

\subsubsection*{Comparison with related works}

We briefly give more details concerning the propagation of two wave
packets described in \cite{CaCa11, CaFe11}. Since both cases are rather
similar, we describe the case of a Hartree nonlinearity only
(\cite{CaCa11}). The main difference in the setting is that
\eqref{eq:r3alpha} is replaced with 
\begin{equation*}
  i\eps \d_t \psi^\eps + \frac{\eps^2}{2}\Delta \psi^\eps =
  V(t,x)\psi^\eps +\lambda\eps^\alpha\(|x|^{-\gamma}\ast 
  |\psi^\eps|^2\) \psi^\eps,\quad t\ge 0,\ x\in \R^d,
\end{equation*}
with $\lambda\in \R$ and $0<\gamma<\min (2,d)$. The critical value for
$\alpha$ is $\alpha_c=1+\gamma/2>1$. When
$\alpha=\alpha_c$, the propagation of one initial wave packet is well
approximated by 
\begin{equation*}
  \varphi^\eps(t,x) =\eps^{-d/4}u\(t,
\frac{x-q(t)}{\sqrt{\eps}}\)e^{i(S(t) 
+p(t)\cdot(x-q(t)))/\eps},
\end{equation*}
where $(q,p)$ is given by \eqref{eq:traj}, $S$ is the classical action
\eqref{eq:classicalaction}, and the envelope $u$ solves the nonlinear
equation
\begin{equation*}
  i\d_t u +\frac{1}{2}\Delta u = \frac{1}{2}\<y,\nabla^2V\(t,q(t)\)y\>u
  +\lambda \(|y|^{-\gamma}\ast |u|^2\)u.
\end{equation*}
It is proved that two such wave packets evolve independently from each
other, up to an error which is
$\O\(\eps^{\frac{\gamma}{2(1+\gamma)}}\)$. A way to understand this
result compared to the ones presented here is that since $\alpha_c>1$,
no interference can occur at leading order. 
\smallbreak

There are several results which may seem similar to ours, in the case
of one initial wave packet: see e.g. 
\cite{BJ00,FrTsYa02,GaSi07,HoZw07,JFGS06,Keraani02,KeraaniAA}. In those
papers, the initial amplitude 
$a$ is very specific, since it is a ground state. The propagation and
stability of multi-solitons for the nonlinear Schr\"odinger equation
(without external potential) have been studied in
\cite{MaMe06,MaMeTs06,Per04,RoScSo05} (see also
\cite{Tao09}). In the framework of these papers, the waves do
not interfere.  
\smallbreak

In \cite{ASFS09}, a problem which shares several features with ours is
studied: there is an
external potential, the regime is semiclassical (see \cite{HoZw07}), and
nonlinear. The envelopes of the initial data are two
solitons. The structure 
of the soliton manifold implies some rigidity on the evolution of the
initial data. Eventually, the two waves do not interact at leading
order. 
\smallbreak

On the other hand, in \cite{KrMaRa09}, the case of two solitons for the
Hartree equation has been studied. In this non-semiclassical setting,
and in the absence of an external potential, the authors construct a
solution which behaves, for large time, like the sum of two solitary
waves, whose respective centers in phase space evolve according to the
two-body problem. This feature can be compared to
Theorem~\ref{theo:alpha0} (with $V=0$), where the centers of the wave packets
evolve according to the nonlinear system
\eqref{eq:trajmodif}. Nevertheless, the envelopes are given by the
ground state, and do not evolve with time. The analytical approach is
different: in \cite{KrMaRa09}, a fine study of the Hartree operator
linearized about the soliton is performed, in particular to understand
the spectral properties of this operator. On the other hand, we do not
consider such an operator; a similar approach with general profiles $a_1,a_2$
like we consider would probably be out of reach. 
\smallbreak

In \cite{GrMaSa02,Sa05,Sa09}, a semiclassical regime is studied, in
the presence of an external potential and a nonlinearity. The
potential is a double well potential, and the associated
Hamiltonian  has two
eigenfunctions. For initial data carried by these two eigenfunctions,
it is shown that the nonlinear solution may remain concentrated on the
eigenfunctions, with time-dependent coefficients which interact
nonlinearly. 
\smallbreak

In all the cases mentioned above, the nonlinear interference of the
envelopes is negligible, due to the fact that the envelopes decay
exponentially. In our case, the decay may be much weaker
(algebraic). However, even 
though we have seen that the envelopes always interact nonlinearly in
the cases studied here, we will see that some ``rectangle'' terms are
negligible in the limit $\eps\to 0$, thanks to a microlocal argument
(see Section~\ref{sec:mixed}). 
\smallbreak

We finally point out
  that nonlinear interactions of amplitudes have been analyzed in the
  context of weakly nonlinear geometric optics for Schr\"odinger or
  Hartree equations in various contexts (not to mention the even wider
  literature concerning hyperbolic equations); see for instance
  \cite{CDS10,GMS08}.

\subsubsection*{Notations}

Throughout this paper, $\R_+$ stands for $[0,\infty)$. We also use the
standard convention, for $A\in \R^n$, $n\ge 1$,
\begin{equation*}
  \<A\>=\sqrt{1+|A|^2}. 
\end{equation*}
For two positive numbers $a^\eps$ and $b^\eps$,
the notation $ a^\eps\lesssim b^\eps$
means that there exists $C>0$ \emph{independent of} $\eps$ such that
for all $\eps\in (0,1]$, $a^\eps\le Cb^\eps$.  

\section{Formal derivation}
\label{sec:rough}

We resume the same approach as in the case of a single wave packet
(\cite{CaCa11}), in the case of \eqref{eq:ci2}: from this point of
view the computations below include the ones presented in
\cite{CaCa11}. 

\subsection{The general strategy}

We seek an approximate solution of the form
\begin{equation}
  \label{eq:approx}
  \psi_{\rm app}^\eps(t,x)=\eps^{-d/4} \sum_{j=1,2}u_j 
\left(t,\frac{x-q_j(t)}{\sqrt\eps}\right)e^{i\left(S_j(t)+p_j(t)\cdot
    (x-q_j(t))\right)/\eps},
\end{equation}
for some profiles $u_j$ independent of $\eps$, and some functions $S_j(t)$
to be determined. These functions $S_j$ correspond to the classical
action \eqref{eq:classicalaction} in the linear case. We will see that
according to the value $\alpha$ in \eqref{eq:r3alpha},
the expression of $S_j$ may vary, accounting for nonlinear effects due to the
presence of the Hartree nonlinearity, and so it may be convenient to
consider $\eps$-dependent functions $S_j$.  Also, according to the
value of $\alpha$, the pairs $(q_j,p_j)$ will solve the standard
Hamiltonian system \eqref{eq:traj}, or a modified one. 
Denote 
\begin{equation*}
  \phi_j(t,x) = S_j(t)+p_j(t)\cdot    \(x-q_j(t)\).
\end{equation*}
In the cases $\alpha=0,1/2$ and $\alpha=1$ considered in this paper,
we will see that we can write 
\begin{equation}\label{eq:generalstrategy}
  \begin{aligned}
   i\eps &\d_t \psi_{\rm app}^\eps + \frac{\eps^2}{2}\Delta \psi_{\rm app}^\eps -
   V\psi_{\rm app}^\eps -\eps^\alpha\(K\ast 
  |\psi_{\rm app}^\eps|^2\)\psi_{\rm app}^\eps  = \\
&\eps^{-d/4}\sum_{j=1,2}e^{i\phi_j(t,x)/\eps}
\(b_{0j}+\sqrt \eps b_{1j} + \eps b_{2j} + \eps
r_j^\eps\)\left(t,\frac{x-q_j(t)}{\sqrt\eps}\right), 
  \end{aligned}
\end{equation}
for $b_{ij}$ independent of $\eps$. The
approximate solution 
$\psi_{\rm app}^\eps$ is determined by the conditions
\begin{equation*}
  b_{0j}=b_{1j}=b_{2j}=0,\quad j=1,2.
\end{equation*}
The remaining factor $r_j^\eps$ accounts for the error between the
exact solution $\psi^\eps$ and the approximate solution
$\psi_{\rm app}^\eps$. Introduce two new space variables, which are
naturally associated to each of the two approximating wave packets:
\begin{equation*}
  y_j=\frac{x-q_j(t)}{\sqrt\eps},\quad j=1,2.
\end{equation*}
At this stage, the introduction of these variables may seem very
artificial, since only the $x$ variable will eventually remain. It can
be understood as a change of variable corresponding to the moving
frame of each wave packet. Technically, it will be justified by the
fact, already present in the linear case $K=0$, that the remainders
$r_j^\eps$ will satisfy pointwise estimates of the form 
\begin{equation*}
  \left|r_j^\eps\left(t,\frac{x-q_j(t)}{\sqrt\eps}\right)\right|\lesssim
\sqrt\eps \<y_j\>^3 A_j^\eps(t,y_j)\Big|_{y_j= \frac{x-q_j(t)}{\sqrt\eps}}.
\end{equation*}
The functions $A_j^\eps$ are well localized, in the sense that
$y_j\mapsto \<y_j\>^kA_j^\eps(t,y_j)$ in bounded in $L^2(\R^d)$ at
least for $k=3$ (but possibly for larger $k$'s), while typically, a
function of the form
\begin{equation*}
  \<y_j+\eta \frac{q_1(t)-q_2(t)}{\sqrt\eps}\>^3
  A_j^\eps(t,y_j)
\end{equation*}
cannot be controlled in $L^2(\R^d)$ uniformly in $\eps$ and $t\in [0,T]$ if
$\eta\not =0$. 
\smallbreak

To conclude this subsection, we expand each term on
the left hand side 
of \eqref{eq:generalstrategy} so it has the form of the right hand
side. In the following subsections, we discuss the
outcome according 
to the value $\alpha=0,1/2$ or $\alpha=1$.
\smallbreak

The linear terms are computed as follows:
\begin{align*}
  i\eps\d_t \psi_{\rm app}^\eps &=
  \eps^{-d/4}\sum_{j=1,2}e^{i\phi_j(t,x)/\eps}\(i\eps \d_t 
  u_j - i\sqrt \eps \dot q_j(t)\cdot \nabla u_j - u_j\d_t \phi_j \).\\
\frac{\eps^2}{2}\Delta \psi_{\rm app}^\eps &=
\eps^{-d/4}\sum_{j=1,2}e^{i\phi_j(t,x)/\eps}\( \frac{\eps}{2}\Delta u_j +i\sqrt\eps
p_j(t)\cdot \nabla u_j -\frac{|p_j(t)|^2}{2}u_j\). 
\end{align*}
Here, as well as below, one should remember that the functions are
assessed as in \eqref{eq:approx}. Recalling that the relevant space
variable for $u_j$ is  $y_j$,
we have:
\begin{equation*}
  \d_t \phi_j = \dot S_j(t) + \frac{d}{dt}\(p_j(t)\cdot (x-q_j(t))\) =
  \dot S_j(t) +\sqrt \eps \dot p_j(t)\cdot y_j-p_j(t)\cdot \dot q_j(t).
\end{equation*}
For the linear potential term, we write
\begin{align*}
  V\psi_{\rm app}^\eps &=
  V(t,x)\eps^{-d/4}\sum_{j=1,2}e^{i\phi_j(t,x)/\eps}u_j\(t,y_j\)\\
&=
\eps^{-d/4}\sum_{j=1,2}e^{i\phi_j(t,x)/\eps}
V\(t,q_j(t)+y_j\sqrt\eps\) u_j\(t,y_j\), 
\end{align*}
and we perform a Taylor expansion for $V$ about $x=q_j(t)$:
\begin{align*}
  V\(t,q_j(t)+y_j\sqrt\eps\)u_j(t,y_j) &= V\(t,q_j(t)\) u_j(t,y_j)+ \sqrt \eps 
y_j\cdot \nabla V\(t,q_j(t)\)u_j(t,y_j)\\
&\quad  + \frac{\eps}{2}\<y_j,\nabla^2
V\(t,q_j(t)\)y_j\>u_j(t,y_j) 
+ \eps^{3/2} r_{jV}^\eps(t,y_j),
\end{align*}
with
\begin{equation}\label{eq:estrjV}
  |r_{jV}^\eps(t,y_j)|\le C\<y_j\>^3|u_j(t,y_j)|,
\end{equation}
for some $C$ independent of $\eps$, $t$ and $y_j$, in view of
Assumption~\ref{hyp:gen}. 
In the case $K=0$, we come up with the relations:
\begin{align*}
  b_{0j}^{\rm lin} &= -u_j\( \dot S_j(t)-p_j(t)\cdot \dot q_j(t)
  +\frac{|p_j(t)|^2}{2}+V\(t,q_j(t)\)\) .\\
b_{1j}^{\rm lin}&= -i\( \dot q_j(t)-p_j(t)\)\cdot \nabla u_j-y_j\cdot
\(\dot p_j(t)+ 
\nabla V\(t,q_j(t)\)\) u_j.\\
b_{2j}^{\rm lin} &=  i\d_t u_j +\frac{1}{2}\Delta u_j -
\frac{1}{2}\<y_j,\nabla^2 V\(t,q_j(t)\)y_j\>u_j. 
\end{align*}
For the nonlinear term, the computations are heavier:
\begin{equation*}
  \(K\ast
  |\psi_{\rm app}^\eps|^2\)\psi_{\rm app}^\eps=
  \eps^{-d/4}\sum_{j=1,2}e^{i\phi_j(t,x)/\eps} \(\int
  K(z)|\psi_{\rm app}^\eps(t,x-z)|^2dz\) u_j(t,y_j).
\end{equation*}
Eventually, each envelope $u_j$ will solve a Schr\"odinger equation,
the two equations being coupled. The precise expression of these
equations depends on $\alpha$, but at this stage, we can notice that
for $j=1,2$, $u_j$ solves an equation of the form
\begin{equation}
  \label{eq:ujgen}
  i\d_t u_j +\frac{1}{2}\Delta u_j =
\frac{1}{2}\<y_j,\nabla^2 V\(t,q_j(t)\)y_j\>u_j + F_j u_j,
\end{equation}
where the function $F_j$, accounting for nonlinear effects due to the
Hartree kernel, is \emph{real-valued}. We infer an important property:
the $L^2$-norm of $u_j$ is independent of time,
\begin{equation}\label{eq:masse}
  \|u_j(t)\|_{L^2(\R^d)}=\|a_j\|_{L^2(\R^d)},\quad \forall t\ge 0, \
  j=1,2. 
\end{equation}
At this stage, this is only a formal remark. 
\smallbreak

In the above sum, the variable
$x$ must be expressed in terms of  $y_j$:
\begin{align*}
  K\ast
  |\psi_{\rm app}^\eps|^2&=\int K\(z\)\left\lvert
    \apsi\(t,q_j(t)+\sqrt\eps y_j -z\)\right\rvert^2dz\\
&=\eps^{-d/2}\int K\(z\)\left\lvert
    \sum_{k=1,2}e^{i\phi_k(t,x-z)/\eps}u_k\(t, y_j
    +\frac{q_j(t)-q_k(t)}{\sqrt\eps} 
    -\frac{z}{\sqrt\eps}\)\right\rvert^2dz. 
\end{align*}
Before changing the integration variable, we develop the squared
modulus:
\begin{align*}
  &\left\lvert\sum_{k=1,2}e^{i\phi_k(t,x-z)/\eps}u_k\(t, y_j
    +\frac{q_j(t)-q_k(t)}{\sqrt\eps} 
    -\frac{z}{\sqrt\eps}\)\right\rvert^2=\\
  &\left\lvert u_1\(t, y_j 
    +\frac{q_j(t)-q_1(t)}{\sqrt\eps} 
    -\frac{z}{\sqrt\eps}\)\right\rvert^2+\left\lvert u_2\(t, y_j 
    +\frac{q_j(t)-q_2(t)}{\sqrt\eps} 
    -\frac{z}{\sqrt\eps}\)\right\rvert^2 \\
&+2\RE e^{i\(\phi_1-\phi_2\)/\eps}u_1\(t, y_j 
    +\frac{q_j(t)-q_1(t)}{\sqrt\eps} 
    -\frac{z}{\sqrt\eps}\) \overline u_2\(t, y_j 
    +\frac{q_j(t)-q_2(t)}{\sqrt\eps} 
    -\frac{z}{\sqrt\eps}\),
\end{align*}
where $\phi_1-\phi_2$ stands for $\phi_1(t,x-z)-\phi_2(t,x-z)$. To
ease notations, we shall denote in the rest of this paper:
\begin{equation*}
  \delta q(t)=q_1(t)-q_2(t);\quad \delta p(t)= p_1(t)-p_2(t). 
\end{equation*}
We can write
\begin{equation*}
  \(K\ast
  |\psi_{\rm app}^\eps|^2\)\psi_{\rm app}^\eps=
  \eps^{-d/4}\sum_{j=1,2}e^{i\phi_j(t,x)/\eps} V^{\rm NL}_j(t,y_j) u_j(t,y_j),
\end{equation*}
with 
\begin{align*}
  V^{\rm NL}_1(t,y_1) &= \eps^{-d/2}\int K(z)\Bigg(\left\lvert u_1\(t, y_1 
    -\frac{z}{\sqrt\eps}\)\right\rvert^2+\left\lvert u_2\(t, y_1 
    +\frac{\delta q(t)}{\sqrt\eps} 
    -\frac{z}{\sqrt\eps}\)\right\rvert^2 \\
&+2\RE e^{i\(\phi_1-\phi_2\)/\eps}u_1\(t, y_1
    -\frac{z}{\sqrt\eps}\) \overline u_2\(t, y_1 
    +\frac{\delta q(t)}{\sqrt\eps} 
    -\frac{z}{\sqrt\eps}\)\Bigg)dz,\\
V^{\rm NL}_2(t,y_2) &= \eps^{-d/2}\int K(z)\Bigg(\left\lvert u_1\(t,
  y_2-\frac{\delta q(t)}{\sqrt\eps}
    -\frac{z}{\sqrt\eps}\)\right\rvert^2+\left\lvert u_2\(t, y_2 
    -\frac{z}{\sqrt\eps}\)\right\rvert^2 \\
&+2\RE e^{i\(\phi_1-\phi_2\)/\eps}u_1\(t, y_2-\frac{\delta q(t)}{\sqrt\eps}
    -\frac{z}{\sqrt\eps}\) \overline u_2\(t, y_2 
    -\frac{z}{\sqrt\eps}\)\Bigg)dz.
\end{align*}
Each nonlinear potential $V^{\rm NL}_j$ is the sum of three
terms. The third term in each of these two 
  expressions, involving the product $u_1\overline u_2$, will be  referred to as
  \emph{rectangle term}, as opposed to \emph{squared terms}, involving
  squared moduli. The two rectangle terms are examined  in
  Section~\ref{sec:mixed}, where we show that 
at least on finite time intervals, they are negligible in the limit
$\eps\to 0$, regardless of the value of $\alpha$. Therefore, we now
consider only the squared terms. 
Changing variables in the
integrations and performing a Taylor expansion of the kernel $K$, we
find successively (recall that $G_j$ is
  defined by \eqref{def:Gj}):
\begin{align*}
  \eps^{-d/2}\int K(z)& \left\lvert u_1\(t, y_1 
    -\frac{z}{\sqrt\eps}\)\right\rvert^2dz =\int
  K\(\sqrt\eps(y_1-z)\)\left\lvert u_1\(t, z\)\right\rvert^2dz \\
&= K(0)\|a_1\|_{L^2}^2 + \sqrt \eps \|a_1\|_{L^2}^2 y_1\cdot \nabla
K(0) -\sqrt\eps \nabla K(0)\cdot G_1(t) \\
&+\frac{\eps}{2}
\<y_1,\nabla^2K(0)y_1\> \|a_1\|_{L^2}^2+ \frac{\eps}{2} \int
\<z,\nabla^2K(0)z\>|u_1(t,z)|^2dz\\ 
&- \eps \<\nabla^2K(0)G_1(t),y_1\>+\eps^{3/2}\int r_{11}^\eps
(t,z-y_1)|u_1(t,z)|^2dz, 
\end{align*}
\begin{align*}
  \eps^{-d/2}\int K(z)& \left\lvert u_2\(t, y_1 
    +\frac{\delta q(t)}{\sqrt\eps} 
    -\frac{z}{\sqrt\eps}\)\right\rvert^2dz \\
&=\int
  K\(\delta q(t)+\sqrt\eps(y_1-z)\)\left\lvert u_2\(t, z\)\right\rvert^2dz \\
&= K(\delta q)\|a_2\|_{L^2}^2 + \sqrt \eps \|a_2\|_{L^2}^2 y_1\cdot \nabla
K(\delta q) -\sqrt\eps \nabla K(\delta q)\cdot G_2(t) \\
&+\frac{\eps}{2}
\<y_1,\nabla^2K(\delta q)y_1\> \|a_2\|_{L^2}^2+ \frac{\eps}{2} \int
\<z,\nabla^2K(\delta q)z\>|u_2(t,z)|^2dz\\ 
&- \eps \<\nabla^2K(\delta q)G_2(t),y_1\>+\eps^{3/2}\int r_{12}^\eps
(t,z-y_1)|u_2(t,z)|^2dz,
\end{align*}

\begin{align*}
 \eps^{-d/2}\int K(z)&\left\lvert u_1\(t,
  y_2-\frac{\delta q(t)}{\sqrt\eps}
    -\frac{z}{\sqrt\eps}\)\right\rvert^2 dz\\
&= K(-\delta q)\|a_1\|_{L^2}^2 + \sqrt \eps \|a_1\|_{L^2}^2 y_2\cdot \nabla
K(-\delta q) -\sqrt\eps \nabla K(-\delta q)\cdot G_1(t) \\
&+\frac{\eps}{2}
\<y_2,\nabla^2K(-\delta q)y_2\> \|a_1\|_{L^2}^2+ \frac{\eps}{2} \int
\<z,\nabla^2K(-\delta q)z\>|u_1(t,z)|^2dz\\ 
&- \eps \<\nabla^2K(-\delta q)G_1(t),y_2\>+\eps^{3/2}\int r_{21}^\eps
(t,z-y_2)|u_1(t,z)|^2dz, 
\end{align*}
\begin{align*}
  \eps^{-d/2}\int K(z)& \left\lvert u_2\(t, y_2 
    -\frac{z}{\sqrt\eps}\)\right\rvert^2dz \\
&= K(0)\|a_2\|_{L^2}^2 + \sqrt \eps \|a_2\|_{L^2}^2 y_2\cdot \nabla
K(0) -\sqrt\eps \nabla K(0)\cdot G_2(t) \\
&+\frac{\eps}{2}
\<y_2,\nabla^2K(0)y_2\> \|a_2\|_{L^2}^2+ \frac{\eps}{2} \int
\<z,\nabla^2K(0)z\>|u_2(t,z)|^2dz\\ 
&- \eps \<\nabla^2K(0)G_2(t),y_2\>+\eps^{3/2}\int r_{22}^\eps
(t,z-y_2)|u_2(t,z)|^2dz,
\end{align*}
where the functions $r_{jk}^\eps$ satisfy uniform estimates of the
form
\begin{equation}\label{eq:estrij}
  |r_{jk}^\eps(t,z)|\le C(T)  \<z\>^3 ,\quad \forall z\in
  \R^d, t\in [0,T],
\end{equation}
with $C(T)$ independent of $\eps$, $j$ and $k$, but possibly depending
on $T$.

\subsection{The critical case: $\alpha=1$}
\label{sec:alpha1}
When $\alpha>1$, we have $b_{\ell j}=b_{\ell j}^{\rm lin}$ for all
$\ell, j$: nonlinear effects are not present at leading order. 
When $\alpha=1$, we still have $b_{\ell j}=b_{\ell j}^{\rm lin}$ for
$\ell=0,1$: we still demand $(q_j,p_j)$ to solve \eqref{eq:traj} in
order for the equations $b_{\ell j}=0$, $\ell=0,1$, to be satisfied, and
$S_j$ is defined as in \eqref{eq:classicalaction}. On
the other hand, the expression for $b_{2j}$ is altered:
\begin{align*}
  b_{21} &= i\d_t u_1 +\frac{1}{2}\Delta u_1 -
\frac{1}{2}\<y_1,\nabla^2 V\(t,q_1(t)\)y_1\>u_1-K(0)\|a_1\|_{L^2}^2
u_1\\
& -K\(\delta q(t)\) \|a_2\|_{L^2}^2 u_1,\\
b_{22}& = i\d_t u_2 +\frac{1}{2}\Delta u_2 -
\frac{1}{2}\<y_2,\nabla^2 V\(t,q_2(t)\)y_2\>u_2-K(0)\|a_2\|_{L^2}^2
u_2\\
& -K\(-\delta q(t)\) \|a_1\|_{L^2}^2 u_2.
\end{align*}
The last term in each expression accounts for a coupling, revealing a
leading order interaction of the two wave packets. This coupling can
be understood rather explicitly, since it consists of a purely time
dependent potential. Solving the equations $b_{2j}=0$, we
infer, with obvious notations adapted from  \eqref{eq:linearu}, 
\begin{align*}
  u_1(t,y_1)&=u^{\rm lin}_1(t,y_1)\exp\(-i t
    K(0)\|a_1\|_{L^2}^2-i\|a_2\|_{L^2}^2 \int_0^t K\(\delta
    q(s)\)ds\),\\
u_2(t,y_2)&=u^{\rm lin}_2(t,y_2)\exp\(-i t
    K(0)\|a_2\|_{L^2}^2-i\|a_1\|_{L^2}^2 \int_0^t K\(-\delta
    q(s)\)ds\).
\end{align*}
The presence of these phase shifts accounts for 
nonlinear effects at leading order in the approximate wave packet
$\psi_{\rm app}^\eps$: nonlinear effects in the case of a single wave
packet, and nonlinear coupling, since we
assume $\|a_j\|_{L^2}\not =0$. For the remainder terms, we have  the (rough)
pointwise estimate 
\begin{equation}\label{eq:reste}
  |r^\eps_{j}(t,y_j)|\le C(T)\sqrt\eps \<y_j\>^3|u_j(t,y_j)|\(1+
  \sum_{k =1,2}\|u_k(t)\|_{\Sigma^2}^2\), \quad t\in [0,T].
\end{equation}
The remainder $r_j^\eps$ is the sum of the terms
  $r_{jV}^\eps$ and  $\eps^{\alpha+3/2}(r_{jk}^\eps\ast |u_k|^2)u_j$, $k=1,2$,
  so this estimate is an easy 
consequence of \eqref{eq:estrjV} and \eqref{eq:estrij}.
To be precise, this estimate is valid up to the rectangle terms that we have
discarded so far, when we have developed $(K\ast
  |\psi_{\rm app}^\eps|^2)\psi_{\rm app}^\eps$. We will 
see in Section~\ref{sec:mixed} that they satisfy a similar estimate (see
Corollary~\ref{cor:reste}).

\subsection{The case $\alpha=1/2$} 
\label{sec:alpha12}
We still have $b_{0j}=b_{0j}^{\rm lin}$, but now
\begin{align*}
  b_{11}&= -i\( \dot q_1(t)-p_1(t)\)\cdot \nabla u_1-y_1\cdot \(\dot p_1(t)+
\nabla V\(t,q_1(t)\)\) u_1 \\
&\quad -K(0)\|a_1\|_{L^2}^2 u_1-K\(\delta
q\)\|a_2\|_{L^2}^2 u_1, \\
b_{12}& = -i\( \dot q_2(t)-p_2(t)\)\cdot \nabla u_2-y_2\cdot \(\dot p_2(t)+
\nabla V\(t,q_2(t)\)\) u_2 \\
&\quad -K(0)\|a_2\|_{L^2}^2 u_2-K\(-\delta
q\)\|a_1\|_{L^2}^2 u_2, \\
b_{21} & = i\d_t u_1 +\frac{1}{2}\Delta u_1 -
\frac{1}{2}\<y_1,\nabla^2 V\(t,q_1(t)\)y_1\>u_1- \|a_1\|_{L^2}^2
y_1\cdot \nabla K(0) u_1\\
&\quad - \|a_2\|_{L^2}^2
y_1\cdot \nabla K(\delta q) u_1
+ \nabla K(0)\cdot G_1(t) u_1 + 
\nabla K(\delta q)\cdot G_2(t) u_1,\\
b_{22} & = i\d_t u_2 +\frac{1}{2}\Delta u_2 -
\frac{1}{2}\<y_2,\nabla^2 V\(t,q_2(t)\)y_2\>u_2- \|a_2\|_{L^2}^2
y_2\cdot \nabla K(0) u_2\\
&\quad - \|a_1\|_{L^2}^2
y_2\cdot \nabla K(-\delta q) u_2
+ \nabla K(0)\cdot G_2(t) u_2 + 
\nabla K(-\delta q)\cdot G_1(t) u_2 . 
\end{align*}
The first line in $b_{1j}$ is zero if $(q_j,p_j)$ is the classical
trajectory given by \eqref{eq:traj}. On the other hand, it does not seem
to be possible to cancel out the second line in $b_{1j}$, even by
modifying \eqref{eq:traj}: we have three sets of terms, involving
$\nabla u_1$, $y_1 u_1$ and $u_1$, respectively, so they must be
treated separately.
As in \cite{CaCa11}, we then modify the general strategy, and allow
$b_{0j}$ to depend on $\eps$. We alter the hierarchy as follows:
\begin{align*}
  b_{01}^\eps&= -u_1\Big( \dot S_1(t)-p_1(t)\cdot \dot q_1(t)
  +\frac{|p_1(t)|^2}{2}+V\(t,q_1(t)\)\\
&\phantom{= -u_1\Big(}+\sqrt\eps
  K(0)\|a_1\|_{L^2}^2+\sqrt\eps K\(\delta
q\)\|a_2\|_{L^2}^2 \Big),\\
 b_{02}^\eps&= -u_2\Big( \dot S_2(t)-p_2(t)\cdot \dot q_2(t)
  +\frac{|p_2(t)|^2}{2}+V\(t,q_2(t)\)\\
&\phantom{= -u_1\Big(}+\sqrt\eps
  K(0)\|a_2\|_{L^2}^2+\sqrt\eps K\(-\delta
q\)\|a_1\|_{L^2}^2 \Big),\\
  b_{11}&= -i\( \dot q_1(t)-p_1(t)\)\cdot \nabla u_1-y_1\cdot \(\dot p_1(t)+
\nabla V\(t,q_1(t)\)\) u_1 , \\
b_{12}& = -i\( \dot q_2(t)-p_2(t)\)\cdot \nabla u_2-y_2\cdot \(\dot p_2(t)+
\nabla V\(t,q_2(t)\)\) u_2 , 
\end{align*}
and we leave $b_{2j}$ unchanged. Like before, $b_{1j}=0$ provided that
$(q_j,p_j)$ solves \eqref{eq:traj}. The novelty is that we now
consider  modified, $\eps$-dependent, actions:
\begin{equation*}
\left\{
  \begin{aligned}
S^\eps_1(t) &= \int_{0}^{t}\(\frac{1}{2}
    |p_1(s)|^2-V(s,q_1(s))\)ds\\
&\phantom{=}-t\sqrt{\eps} K(0)\|a_1\|_{L^2(\R^d)}^2
    -\sqrt\eps \|a_2\|_{L^2(\R^d)}^2\int_0^t K\(\delta q(s)\)ds ,\\
S^\eps_2(t) &= \int_{0}^{t}\(\frac{1}{2}
    |p_2(s)|^2-V(s,q_2(s))\)ds\\
&\phantom{=}-t\sqrt{\eps} K(0)\|a_2\|_{L^2(\R^d)}^2
    -\sqrt\eps \|a_1\|_{L^2(\R^d)}^2\int_0^t K\(-\delta q(s)\)ds.
\end{aligned}
\right.
\end{equation*}
These expressions are exactly those given in the introduction
\eqref{eq:actionhalf}. 
The equations $b_{2j}=0$ are envelope
equations, which are nonlinear since $G_k$ is a nonlinear function of
$u_k$. Note however that $G_k$ yields a purely time-dependent
potential. Consider the solution to 
\begin{align*}
i\d_t \tilde u_1 +\frac{1}{2}\Delta \tilde u_1 & =
\frac{1}{2}\<y_1,\nabla^2 V\(t,q_1(t)\)y_1\>\tilde u_1+\|a_1\|_{L^2}^2
y_1\cdot \nabla K(0) \tilde u_1\\
&\phantom{=}+ \|a_2\|_{L^2}^2
y_1\cdot \nabla K(\delta q) \tilde u_1,\\
i\d_t \tilde u_2 +\frac{1}{2}\Delta \tilde u_2 &=
\frac{1}{2}\<y_2,\nabla^2 V\(t,q_2(t)\)y_2\>\tilde u_2+ \|a_2\|_{L^2}^2
y_2\cdot \nabla K(0) \tilde u_2\\
&\phantom{=} + \|a_1\|_{L^2}^2
y_2\cdot \nabla K(-\delta q) \tilde u_2.
\end{align*}
Set
\begin{align*}
&  u_1(t,y_1) = \tilde u_1(t,y_1)\exp\(i\int_0^t\(\nabla K(0)\cdot
  \tilde G_1(s) + \nabla K\(\delta q(s)\)\cdot \tilde G_2(s)\)ds \),\\
& u_2(t,y_2) = \tilde u_2(t,y_2)\exp\(i\int_0^t\(\nabla K(0)\cdot
  \tilde G_2(s) + \nabla K\(-\delta q(s)\)\cdot \tilde G_1(s)\)ds \),
\end{align*}
where 
$$ \tilde G_j(t) = \int_{\R^d}z|\tilde u_j(t,z)|^2dz.$$
It is clear that $|u_j|=|\tilde u_j|$, hence $\tilde G_j=G_j$,
and so $u_1$ and $u_2$ are such that $b_{21}=b_{22}=0$, and correspond
to the envelopes introduced in \S\ref{sec:super12}. 
 Finally, we still have a
remainder term satisfying \eqref{eq:reste} (up to the terms treated in
\S\ref{sec:mixed}). 

\subsection{The case $\alpha=0$} Now all the coefficients
$b_{\ell j}$ are affected by the nonlinearity:
\begin{align*}
  b_{01}&= -u_1\Big( \dot S_1(t)-p_1(t)\cdot \dot q_1(t)
  +\frac{|p_1(t)|^2}{2}+V\(t,q_1(t)\)+K(0)\|a_1\|_{L^2}^2\\
&\phantom{=-u_1\Big(}+K\(\delta
  q\)\|a_2\|_{L^2}^2 \Big), \\
  b_{02}&= -u_2\Big( \dot S_2(t)-p_2(t)\cdot \dot q_2(t)
  +\frac{|p_2(t)|^2}{2}+V\(t,q_2(t)\)+K(0)\|a_2\|_{L^2}^2\\
&\phantom{=-u_1\Big(}+K\(-\delta
  q\)\|a_1\|_{L^2}^2 \Big), \\
b_{11}&=-i\( \dot q_1(t)-p_1(t)\)\cdot \nabla u_1-y_1\cdot \(\dot p_1(t)+
\nabla V\(t,q_1(t)\)\) u_1 -
   \|a_1\|_{L^2}^2 y_1\cdot \nabla K(0)u_1\\
&\quad - \|a_2\|_{L^2}^2 y_1\cdot \nabla K\(\delta q\)u_1
+ \nabla K(0)\cdot G_1(t) u_1 +\nabla K\(\delta q\)\cdot G_2(t)u_1,\\
b_{12}&=-i\( \dot q_2(t)-p_2(t)\)\cdot \nabla u_2-y_2\cdot \(\dot p_2(t)+
\nabla V\(t,q_2(t)\)\) u_2 -
   \|a_2\|_{L^2}^2 y_2\cdot \nabla K(0)u_2\\
&\quad - \|a_1\|_{L^2}^2 y_2\cdot \nabla K\(-\delta q\)u_2
+ \nabla K(0)\cdot G_2(t) u_2 +\nabla K\(-\delta q\)\cdot G_1(t)u_2,\\
b_{21}&=  i\d_t u_1 +\frac{1}{2}\Delta u_1- \frac{1}{2}\<y_1, M_1(t)y_1\>u_1 +
\<\nabla^2K(0) G_1(t),y_1\>u_1  \\
&\quad +
\<\nabla^2K(\delta q) G_2(t),y_1\>u_1 -\frac{1}{2}\( \int
\<z,\nabla^2K(0)z\>|u_1(t,z)|^2dz\) u_1\\
&\quad-\frac{1}{2} \(\int
\<z,\nabla^2K(\delta q)z\>|u_2(t,z)|^2dz\) u_1, \\
b_{22}&=  i\d_t u_2 +\frac{1}{2}\Delta u_2- \frac{1}{2}\<y_2, M_2(t)y_2\>u_2 +
\<\nabla^2K(0) G_2(t),y_2\>u_2  \\
&\quad +
\<\nabla^2K(-\delta q) G_1(t),y_2\>u_2 -\frac{1}{2}\( \int
\<z,\nabla^2K(0)z\>|u_2(t,z)|^2dz\) u_2\\
&\quad-\frac{1}{2}\( \int
\<z,\nabla^2K(-\delta q)z\>|u_1(t,z)|^2dz\) u_2, 
\end{align*}
where we have denoted 
\begin{align*}
M_1(t)&=\|a_1\|_{L^{2}(\R^{d})}^{2}\nabla^2
K(0)+\|a_2\|_{L^{2}(\R^{d})}^{2}\nabla^2 K\(\delta q(t)\)+
\nabla^2_xV\(t,q_1(t)\),\\
M_2(t)&=\|a_2\|_{L^{2}(\R^{d})}^{2}\nabla^2
K(0)+\|a_1\|_{L^{2}(\R^{d})}^{2}\nabla^2 K\(-\delta q(t)\)+
\nabla^2_xV\(t,q_2(t)\).
\end{align*}
Similar to the case $\alpha=1/2$, we incorporate the last term of
$b_{1j}$ into $b_{0j}$, that is we modify the action as follows: 
\begin{align*}
  S_1^\eps(t) &= \int_0^t \Big(\frac{1}{2}|p_1(s)|^2 - V\(s,q_1(s)\)
 - K(0)\|a_1\|_{L^2}^2- K\(\delta q(s)\)\|a_2\|_{L^2}^2 \\
&\phantom{= \int_0^t \Big(}+\sqrt \eps
  \nabla K(0)\cdot G_1(s)+\sqrt\eps \nabla K\(\delta q(s)\)\cdot
  G_2(s)\Big)ds,\\
     S_2^\eps(t) &= \int_0^t \Big(\frac{1}{2}|p_2(s)|^2 - V\(s,q_2(s)\)
 - K(0)\|a_2\|_{L^2}^2- K\(-\delta q(s)\)\|a_1\|_{L^2}^2 \\
&\phantom{= \int_0^t \Big(}+\sqrt \eps
  \nabla K(0)\cdot G_2(s)+\sqrt\eps \nabla K\(-\delta q(s)\)\cdot
  G_1(s)\Big)ds.
\end{align*}
Note that for $S_j^\eps$ to be well defined, we have to first define
$u_j$, for which we solve the envelope equations, given by $b_{21}=b_{22}=0$.
Canceling the terms $b_{1j}$ yields the modified system of trajectories:
\begin{equation*}
  \left\{
\begin{aligned}
 & \dot q_1(t) = p_1(t),\\
& \dot p_1(t) = -\nabla V\(t,q_1(t)\) -
  \|a_1\|_{L^2}^2 \nabla K(0)- \|a_2\|_{L^2}^2 \nabla
  K\(q_1(t)-q_2(t)\),\\
&\dot q_2(t) = p_2(t),\\
& \dot p_2(t) = -\nabla V\(t,q_2(t)\) -
  \|a_2\|_{L^2}^2 \nabla K(0)- \|a_1\|_{L^2}^2 \nabla
  K\(q_2(t)-q_1(t)\),
\end{aligned}
\right.
\end{equation*}
which is exactly \eqref{eq:trajmodif}.
The remainder term still satisfies \eqref{eq:reste} (up to the terms
treated in  \S\ref{sec:mixed}).
We will examine more carefully the envelope system in \S\ref{sec:refined}.

\section{Analysis of the rectangle interaction term}
\label{sec:mixed}

In the previous section, we have left out the rectangle terms,
claiming that they are negligible in the limit $\eps\to 0$. In this
section, we justify precisely this statement. Since the two terms that
we have discarded are similar, we shall simply consider the first
one:
\begin{equation*}
 2\eps^{-d/2}\RE \int K(z)
  e^{i(\phi_1-\phi_2)(t,x-z)/\eps}u_1\(t,y_1-\frac{z}{\sqrt\eps}\)\overline
  u_2\(t,y_1 +\frac{\delta q(t)}{\sqrt\eps}-\frac{z}{\sqrt\eps}\)dz.
\end{equation*}
Notice that we have not yet expressed the phases $\phi_k$ in terms of
the variable $y_1$, and that the expression of $\phi_k$ varies
according to $\alpha=1$, $\alpha=1/2$, or $\alpha=0$. We shall retain
only a common feature though, that is,
$\phi_k^\eps(t,x)=\Theta_k^\eps(t) + x\cdot p_k(t)$, where only the
purely time dependent function $\Theta$ may depend on $\eps$ (when
$\alpha\in \{1/2,0\}$), and the spatial oscillations are singled
out. Since $x= q_1(t)+\sqrt\eps y_1$, we get, once the real part and
the time oscillations are omitted:
\begin{equation*}
 \eps^{-d/2} \int K(z)
  e^{i\(\sqrt \eps y_1-z\)\cdot \delta
    p(t)/\eps}u_1\(t,y_1-\frac{z}{\sqrt\eps}\)\overline 
  u_2\(t,y_1 +\frac{\delta q(t)}{\sqrt\eps}-\frac{z}{\sqrt\eps}\)dz.
\end{equation*}
Changing the integration variable, and introducing more general
notations, we examine:
\begin{equation}\label{eq:Ieps}
I^\eps(t,y_1)= \int {\mathcal K}\(\sqrt\eps (y_1-z)\)
  e^{iz\cdot \delta
    p(t)/\sqrt\eps}u_1\(t,z\)\overline 
  u_2\(t,z +\frac{\delta q(t)}{\sqrt\eps}\)dz.
\end{equation}
The main result of this section is stated as follows.
\begin{proposition}\label{prop:microloc}
  Let $T>0$. Suppose that $\mathcal K \in W^{\ell,\infty}$, $u_j\in
  C([0,T];\Sigma^k)$ with $k,\ell\in \N$, 
  and consider $I^\eps$ defined by \eqref{eq:Ieps}. There exists 
  $C>0$ independent of $\eps\in (0,1]$, $\mathcal K$, $u_1$ and $u_2$
  such that  
  \begin{equation*}
    \sup_{t\in [0,T]}\|I^\eps(t,\cdot)\|_{L^\infty(\R^d)}\le C
    \eps^{\min(\ell,k)/2}\|\mathcal
    K\|_{W^{\ell,\infty}}\|u_1\|_{L^\infty([0,T];\Sigma^k)} 
\|u_2\|_{L^\infty([0,T];\Sigma^k)}.  
  \end{equation*}
\end{proposition}

In view of the computations performed in Section~\ref{sec:rough},
this result has the following consequence.
\begin{corollary}\label{cor:reste}
  Consider $\apsi$ given by \eqref{eq:approx}, derived in
  Section~\ref{sec:rough}, whose exact expression varies according to
  the cases
  $\alpha=1$, $\alpha=1/2$ or $\alpha=0$. Let $T>0$, and suppose
  $u_j\in C([0,T];\Sigma^3)$. 
Then  $\apsi\in C([0,T];\Sigma^3)$ satisfies $\psi^\eps_{{\rm app}\mid t=0}=
\psi^\eps_{\mid t=0}$ and 
  \begin{equation*}
       i\eps\partial_{t} \apsi+\frac{\eps^{2}}{2} \Delta
\apsi=V\(t,  x\)
\apsi+\sqrt\eps \(K*|\apsi|^{2}
\)\apsi+\eps r^\eps,
  \end{equation*}
where $r^\eps \in C(\R_+;L^2(\R^d))$ is such that there
exists $C$ independent of $\eps$ with
\begin{equation*}
  \sup_{t\in [0,T]}\|r^\eps(t)\|_{L^2(\R^d)}\le C\sqrt\eps. 
\end{equation*}
\end{corollary}
\begin{remark}
  At this stage, the property $u_j\in
  C([0,T];\Sigma^3)$ is established in the cases $\alpha=1$ and
  $\alpha=1/2$. It will require some work to prove it in the case
  $\alpha=0$; see Section~\ref{sec:refined}.  The assumptions of
  Corollary~\ref{cor:reste} are fulfilled, modulo the proof of
  Proposition~\ref{prop:existenv0}. 
\end{remark}
\begin{remark}
 Proposition~\ref{prop:microloc} is a refinement of
 \cite[Proposition~6.3]{CaFe11}, in the sense that the power of $\eps$
 on the right hand side is as large as we wish, provided that 
  $\mathcal K$ is sufficiently smooth, and that the functions
 $u_1$ and $u_2$ are sufficiently localized in space and
 frequency. Identifying precisely the norms of $\mathcal K$,
 $u_1$ and $u_2$, involved in order to get such an error estimate,
 will turn out to be crucial  to
 prove Theorem~\ref{theo:alpha0}, at the level of the bootstrap argument 
 presented in Section~\ref{sec:bootstrap}.
\end{remark}

\subsection{A microlocal property}
\label{sec:microlocal}

The proof of Proposition~\ref{prop:microloc} is based on the following
remark: the function that we integrate is localized away from the
origin in \emph{phase space}:

\begin{lemma}\label{lem:microloc}
 Suppose $(q_{10},p_{10})\not =(q_{20},p_{20})$.   In either of the
 cases $\alpha=1$, $\alpha=1/2$ or $\alpha=0$, the 
  following holds. For any $T>0$, there exists $\eta>0$ such that for
  all $t\in [0,T]$,
  \begin{equation*}
    |\delta q(t)|\ge \eta,\quad \text{or }\lvert \delta p(t)\rvert\ge \eta.
  \end{equation*}
\end{lemma}
\begin{proof}
  We argue by contradiction: if the result were not true, we could
  find a sequence $t_n\in [0,T]$ so that 
  \begin{equation*}
    |\delta q(t_n)| +\lvert \delta p(t_n)\rvert\Tend n \infty 0.
  \end{equation*}
By compactness of $[0,T]$ and continuity of $(q_j,p_j)$, there would
exist $t_*\in [0,T]$ such that
\begin{equation*}
  q_1(t_*)=q_2(t_*),\quad p_1(t_*)=p_2(t_*). 
\end{equation*}
In the cases $\alpha=1$ and $\alpha=1/2$, $(q_j,p_j)$ is given by the
classical Hamiltonian flow \eqref{eq:traj}: uniqueness for
\eqref{eq:traj} implies $(q_{10},p_{10}) =(q_{20},p_{20})$, hence 
a contradiction. 

The case $\alpha=0$ is a bit more delicate, since $(q_j,p_j)$ is no
longer given by a Hamiltonian flow. From \eqref{eq:trajmodif}, we
infer:
\begin{equation*}
  \left\{
\begin{aligned}
  \frac{d(\delta q)}{dt} &= \delta p,\\
\frac{d(\delta p)}{dt}&=\nabla V\(t,q_2(t)\)-\nabla
V\(t,q_1(t)\)+\|a_1\|_{L^2}^2 \(\nabla K\(-\delta
q(t)\)-\nabla K(0)\)\\
&\quad +\|a_2\|_{L^2}^2 \(\nabla K(0)-\nabla K\(\delta
q(t)\)\). 
\end{aligned}
\right.
\end{equation*}
In view of Assumption~\ref{hyp:gen}, there exists $C$ independent of
$t$ such that
\begin{equation*}
  \left| \frac{d(\delta q)}{dt}\right| + \left| \frac{d(\delta
      p)}{dt}\right| \le C \(|\delta p| +|\delta q|\). 
\end{equation*}
Gronwall's Lemma yields a contradiction, and the lemma is proved
in the three cases.
\end{proof}

\subsection{Proof of Proposition~\ref{prop:microloc}}
\label{sec:proofpropmicro}
From Lemma~\ref{lem:microloc}, if suffices to prove the estimate of
Proposition~\ref{prop:microloc} in either of the two cases $|\delta
q(t)|\ge \eta$, or $|\delta p(t)|\ge \eta$. 

\smallbreak

\noindent {\bf First case.} If $|\delta
q(t)|\ge \eta$, we use Cauchy--Schwarz inequality to infer 
\begin{align*}
  |I^\eps(t,y)|&\le \|\mathcal K\|_{L^\infty}\int
  \frac{\<z\>^k}{\<z\>^k}|u_1(t,z)| 
  \frac{\<z+\frac{\delta q(t)}{\sqrt\eps}\>^{k}}{\<z+\frac{\delta
    q(t)}{\sqrt\eps}\>^{k}}\left|u_2\(t,z+\frac{\delta 
    q(t)}{\sqrt\eps}\)\right|dz \\
&\le \|\mathcal K\|_{L^\infty} \|u_1(t)\|_{\Sigma^k}\|u_2(t)\|_{\Sigma^k}
\sup_{z\in \R^d} \<z\>^{-k}\<z+\frac{\delta
    q(t)}{\sqrt\eps}\>^{-k}.
\end{align*}
In view of Peetre inequality (see
e.g. \cite{AlGe07,Treves1}),
\begin{equation*}
  \sup_{z\in \R^d} \<z\>^{-k}\<z+\frac{\delta
    q(t)}{\sqrt\eps}\>^{-k}\le C_k \(\frac{\sqrt\eps}{|\delta
    q(t)|}\)^k\le \frac{C_k}{\eta^k}\eps^{k/2}. 
\end{equation*}
\smallbreak

\noindent {\bf Second case.} If $|\delta p(t)|\ge \eta$, we perform
repeated integrations by parts (like in the standard proof of the nonstationary
phase lemma, see e.g. \cite{AlGe07}) relying on the relation
\begin{equation*}
  e^{iz\cdot \delta
    p(t)/\sqrt\eps} = -i\frac{\sqrt\eps}{|\delta p(t)|^2} \sum_{\ell
    =1}^d (\delta 
    p(t))_\ell\frac{\d }{\d z_\ell}\(e^{iz\cdot \delta
    p(t)/\sqrt\eps}\) . 
\end{equation*}
Note that since we assume $\mathcal K\in W^{\ell,\infty}$ and $u_j\in
\Sigma^k$, we 
perform no more 
than $\min(\ell,k)$ integrations by parts, and Cauchy--Schwarz inequality
yields
\begin{equation*}
  |I^\eps(t,y)| \le \frac{1}{\eta^\ell}\|\mathcal K\|_{W^{\ell,\infty}}
  \|u_1(t)\|_{\Sigma^k}\|u_2(t)\|_{\Sigma^k} \eps^{\min (\ell,k)/2}.
\end{equation*}
The proof of the proposition is complete.

\section{Proof of convergence in the critical case}
\label{sec:proofcritical}

In this section, we complete the proof of
Theorem~\ref{theo:critical}. First, we recall that as a consequence of
\cite{Fujiwara79,Fujiwara}, the system for the envelopes in the linear
case is well-posed in $\Sigma^k$:
\begin{lemma}
  Let $k\in \N$, and $a\in \Sigma^k$. Then \eqref{eq:linearu} has a
  unique solution $u\in C(\R_+;\Sigma^k)$. In addition, the following
  conservation holds:
  \begin{equation*}
    \|u(t)\|_{L^2(\R^d)}=\|a\|_{L^2(\R^d)},\quad \forall t\ge 0. 
  \end{equation*}
\end{lemma}
We infer that if $a_1,a_2\in \Sigma^3$, then
$u_1,u_2$, given by \eqref{eq:envcrit}, belong to
$C(\R_+;\Sigma^3)$. Corollary~\ref{cor:reste} implies that $\apsi$
satisfies
\begin{equation*}
  i\eps\d_t \apsi +\frac{\eps^2}{2}\Delta \apsi = V\apsi +\eps\(K\ast
  |\apsi|^2\)\apsi +\eps r^\eps ;\quad \apsi(0,x)=\psi^\eps(0,x),
\end{equation*}
where the source term $r^\eps$ satisfies:
\begin{equation*}
  \forall T>0,\ \exists C>0,\quad \sup_{t\in
    [0,T]}\|r^\eps(t)\|_{L^2(\R^d)}\le C\sqrt\eps. 
\end{equation*}
Denote by $w^\eps=\psi^\eps-\apsi$ the error term. It satisfies
\begin{align*}
  i\eps\d_t w^\eps +\frac{\eps^2}{2}\Delta w^\eps=Vw^\eps +\eps\( \(K\ast
  |\psi^\eps|^2\)\psi^\eps - \(K\ast
  |\apsi|^2\)\apsi\)-\eps r^\eps,
\end{align*}
with $w^\eps_{\mid t=0}=0$. Writing
\begin{align*}
 \(K\ast
  |\psi^\eps|^2\)\psi^\eps - &\(K\ast
  |\apsi|^2\)\apsi \\
&=  \(K\ast |w^\eps+\apsi|^2\)\(w^\eps+\apsi\)- \(K\ast
  |\apsi|^2\)\apsi \\
&= \(K\ast |w^\eps+\apsi|^2\)w^\eps +\(K\ast\( |w^\eps+\apsi|^2-
|\apsi|^2\)\) \apsi, 
\end{align*}
energy estimates yield, for $t\in [0,T]$, and since
$\psi^\eps,\apsi$ (hence $w^\eps$) are uniformly bounded in
$L^\infty(\R_+;L^2(\R^d))$: 
\begin{align*}
  \|w^\eps(t)\|_{L^2}&\le \int_0^t \left\| K\ast\( |w^\eps+\apsi|^2-
|\apsi|^2\)(s)\right\|_{L^\infty}\|\apsi(s)\|_{L^2}ds \\
& \phantom{\le}+
\int_0^t\|r^\eps(s)\|_{L^2}ds \\
&\le C\int_0^t \left\| \( |w^\eps+\apsi|^2-
|\apsi|^2\)(s)\right\|_{L^1}ds+
\int_0^t\|r^\eps(s)\|_{L^2}ds\\
&\le C\int_0^t \|w^\eps(s)\|_{L^2} + \int_0^t\|r^\eps(s)\|_{L^2}ds,
\end{align*}
for $C>0$ independent of $\eps\in (0,1]$ and $t\ge 0$. 
Theorem~\ref{theo:critical} is then a consequence of Gronwall's
Lemma.

\begin{remark}\label{rem:sur}
  Assuming that we have proved the property $u_j\in
  C([0,T];\Sigma^3)$ in the case $\alpha=0$, which will stem from
  Proposition~\ref{prop:existenv0},  the conclusion of
  Corollary~\ref{cor:reste} holds. However,  
  the estimate given by the above approach is not satisfactory in the
  cases $\alpha=1/2$ and $\alpha=0$. We could prove this way:
\begin{equation*}
  \|\psi^\eps(t) -\apsi(t)\|_{L^2}\le C\sqrt\eps
  e^{Ct/\eps^{1-\alpha}}, \quad t\in [0,T], 
\end{equation*}
for $\alpha=1/2$ and $\alpha=0$, respectively. Contrary to the case
$\alpha=1$ (where Gronwall's Lemma yields a similar estimate), we can
only conclude that $\psi^\eps-\apsi$ is goes to zero on a small time
interval: there exist $c>0$ and $\theta>0$ such that
\begin{equation*}
  \sup_{0\le t\le c \eps^{1-\alpha} |\ln \eps|^\theta}\|\psi^\eps(t)
  -\apsi(t)\|_{L^2}\Tend   \eps 0 0. 
\end{equation*}
Corollary~\ref{cor:reste} is a consistency result, which is not enough
to infer convergence. 
This can be understood as a feature of supercritical regimes: a
different approach is needed, which requires more regularity from $V$,
$K$, and the initial data $a_j$. 
\end{remark}

\section{The envelope equations in the case $\alpha=0$}
\label{sec:refined}

In this section, we prove Proposition~\ref{prop:existenv0}. We first
remark that the last two terms in each equation involved in
\eqref{eq:systenv0} correspond to purely time-dependent potentials,
and can be treated thanks to the gauge transforms
\begin{equation}
  \label{eq:jauge}
  \left\{
    \begin{aligned}
      v_1(t,y)&=u_1(t,y)\exp\Big(-i\int_0^t \int \<z,\nabla^2
      K(0)z\>|u_1(s,z)|^2 dzds \\
&\phantom{=u_1(t,y)\exp\Big(}-i\int_0^t \int \<z,\nabla^2
      K\(\delta q(s)\)z\>|u_2(s,z)|^2 dzds\Big),\\
v_2(t,y)&=u_2(t,y)\exp\Big(-i\int_0^t \int \<z,\nabla^2
      K(0)z\>|u_2(s,z)|^2 dzds \\
&\phantom{=u_1(t,y)\exp\Big(}-i\int_0^t \int \<z,\nabla^2
      K\(-\delta q(s)\)z\>|u_1(s,z)|^2 dzds\Big).
    \end{aligned}
\right.
\end{equation}
Since $K$ is real-valued, we have $|v_j(t,y)|=|u_j(t,y)|$, and
\eqref{eq:jauge} is equivalent to
\begin{equation}
  \label{eq:gaugeinv}
  \left\{
    \begin{aligned}
      u_1(t,y)&=v_1(t,y)\exp\Big(i\int_0^t \int \<z,\nabla^2
      K(0)z\>|v_1(s,z)|^2 dzds \\
&\phantom{=u_1(t,y)\exp\Big(}+i\int_0^t \int \<z,\nabla^2
      K\(\delta q(s)\)z\>|v_2(s,z)|^2 dzds\Big),\\
u_2(t,y)&=v_2(t,y)\exp\Big(i\int_0^t \int \<z,\nabla^2
      K(0)z\>|v_2(s,z)|^2 dzds \\
&\phantom{=u_1(t,y)\exp\Big(}+i\int_0^t \int \<z,\nabla^2
      K\(-\delta q(s)\)z\>|v_1(s,z)|^2 dzds\Big).
    \end{aligned}
\right.
\end{equation}
Formally, $(u_1,u_2)$ solves \eqref{eq:systenv0} if and only if
$(v_1,v_2)$ solves
\begin{equation}
  \label{eq:systv}
  \left\{
\begin{aligned}
  i\d_t  v_1 +\frac{1}{2}\Delta v_1&= \frac{1}{2}\<y, M_1(t)y\>v_1 -
\<\nabla^2K(0) G_1(t),y\>v_1  \\
&\phantom{=} -
\<\nabla^2K\(\delta q(t)\) G_2(t),y\>v_1 , \\
i\d_t  v_2 +\frac{1}{2}\Delta v_2&= \frac{1}{2}\<y, M_2(t)y\>v_2 -
\<\nabla^2K(0) G_2(t),y\>v_2  \\
&\phantom{=}-
\<\nabla^2K\(-\delta q(t)\) G_1(t),y\>v_2 , 
\end{aligned}
\right.
\end{equation}
with the same initial data, $v_{j\mid t=0}=a_j$,
$j=1,2$, where the 
bounded, symmetric matrices $M_1$ and $M_2$ are defined in
\eqref{def:M1} and \eqref{def:M2}, respectively, and where we
have kept the notation
\begin{equation*}
  G_j(t)=\int z|v_j(t,z)|^2dz.
\end{equation*}
Note that the terms involved in the gauge transforms are well defined
when the functions are in $\Sigma^k$ with $k\ge 1$, so
Proposition~\ref{prop:existenv0} stems from the following:
\begin{proposition}\label{prop:v}
  Let $(q_1,p_1,q_2,p_2)$ be given by Lemma~\ref{lem:cauchylip}, and
  $a_1,a_2\in \Sigma^k$ with $k\ge 1$. Then \eqref{eq:systv} has a
  unique solution $(v_1,v_2)\in C(\R_+;\Sigma^k)$ with initial data
  $(a_1,a_2)$. In addition, 
  the following conservations hold:
  \begin{equation*}
    \|v_j(t)\|_{L^2(\R^d)}= \|a_j\|_{L^2(\R^d)},\quad \forall t\ge 0,\
    j=1,2. 
  \end{equation*}
\end{proposition}
\begin{proof}
    The main difficulty is that since the last two terms in each equation
  involve  time dependent potentials which are unbounded in $y$, they
  cannot be treated by perturbative arguments. So to construct a local
  solution, we modify the standard Picard iterative scheme in the same
  fashion as in \cite{CaCa11}, to
  consider
\begin{equation}
  \label{eq:vn}
  \left\{
\begin{aligned}
  i\d_t  v_1^{(n)} +\frac{1}{2}\Delta v_1^{(n)}&= \frac{1}{2}\<y,
  M_1(t)y\>v_1^{(n)} - 
\<\nabla^2K(0) G_1^{(n-1)}(t),y\>v_1^{(n)}  \\
&\phantom{=} -
\<\nabla^2K\(\delta q(t)\) G_2^{(n-1)}(t),y\>v_1^{(n)} , \\
i\d_t  v_2^{(n)} +\frac{1}{2}\Delta v_2^{(n)}&= \frac{1}{2}\<y,
M_2(t)y\>v_2^{(n)} - 
\<\nabla^2K(0) G_2^{(n-1)}(t),y\>v_2^{(n)}  \\
&\phantom{=}-
\<\nabla^2K\(-\delta q(t)\) G_1^{(n-1)}(t),y\>v_2^{(n)} , 
\end{aligned}
\right.
\end{equation}
with $v_{j\mid t=0}^{(n)}=a_j$ for all $n$, $v_j^{(0)}(t,y)=a_j(y)$, and
\begin{equation*}
  G_{j}^{(k)}(t) = \int_{\R^d}z\left|v_j^{(k)}(s,z)\right|^2ds.
\end{equation*}
At each step, we solve  a decoupled system of linear equation, with 
time dependent 
potentials which are at most quadratic in space. If
$G_1^{(n-1)},G_2^{(n-1)}\in L^\infty_{\rm 
  loc}(\R_+)$, \cite{Fujiwara} ensures the existence of
$v_1^{(n)},v_2^{(n)} \in C(\R_+; 
L^{2}(\R^{d}))$.  In addition, we have
\begin{equation*}
  \left\| v_j^{(n)}(t)\right\|_{L^2(\R^d)}=\|a_j\|_{L^2(\R^d)} \quad
  \forall t\ge 0, \ j=1,2. 
\end{equation*}
Applying the operators $y$ and $\nabla$ to
\eqref{eq:vn} yields a closed system of estimates, from which we infer
that $ v_j^{(n)} \in C(\R_+; 
\Sigma)$, hence $G_j^{(n)} \in L^\infty_{\rm  loc}(\R_+)$. Therefore,
the scheme is well-defined. Higher
order regularity can be proven similarly: for $k\ge 1$, by applying
$k$ times the operators $y$ and $\nabla$ to
\eqref{eq:vn}, we check that  $v_1^{(n)},v_2^{(n)}\in C(\R_+;
\Sigma^k)$. As a matter of fact, due to the particular structure of
\eqref{eq:vn},  the only informations needed to prove this property
are $a_j\in \Sigma^k$ and $v_j^{(n-1)}\in C(\R_+; 
\Sigma)$. 
\smallbreak

To prove the
convergence of this scheme we need more precise (uniform in $n$)
estimates. A general computation shows that if $v$ solves
\begin{equation*}
  i\d_t v+\frac{1}{2}\Delta v = \frac{1}{2}\<y,M(t)y\>v+ F(t)\cdot y\,
  v,
\end{equation*}
where $M(t)$ is a real-valued, symmetric matrix, and $F(t)$ is a real-valued
vector, 
then  $\displaystyle G(t)= \int z|v(t,z)|^2dz$ satisfies formally
\begin{align*}
  \dot G(t) &= \IM \int \bar v \nabla v=:J(t),\\
\dot J(t)&= -\int \( M(t)y +F(t)\)|v(t,y)|^2dy =
-M(t)G(t)-F(t)\|v\|_{L^2}^2,
\end{align*}
where the last expression uses implicitly the fact that the $L^2$-norm
of $v$ is independent of time. We have in particular:
\begin{equation*}
  \ddot G(t) +M(t)G(t)=-\|v\|_{L^2}^2F(t).
\end{equation*}
In our case, this yields:
\begin{align}
  \ddot G_1^{(n)} + M_1(t)G_1^{(n)} &=  \|a_1\|_{L^2}^2\nabla^2
  K(0)G_1^{(n-1)}  +\|a_2\|_{L^2}^2\nabla^2
  K\(\delta q(t)\)G_2^{(n-1)},\label{eq:dotG1}\\
\ddot G_2^{(n)} + M_2(t)G_2^{(n)} &=  \|a_2\|_{L^2}^2\nabla^2
  K(0)G_2^{(n-1)}  +\|a_1\|_{L^2}^2\nabla^2
  K\(-\delta q(t)\)G_1^{(n-1)}. \label{eq:dotG2}
\end{align}
Let 
$$f_n(t) = \left|\dot G_1^{(n)}(t)\right|^2 +\left|\dot G_2^{(n)}(t)\right|^2 +
\left|G_1^{(n)}(t)\right|^2 + \left| G_2^{(n)}(t)\right|^2.$$
We have 
\begin{align*}
  \dot f_n(t)&\le 2\sum_{j=1,2}\(\left\lvert\dot G_j^{(n)}(t)\right\rvert
\left\lvert \ddot G_j^{(n)}(t)\right\rvert + \left\lvert\dot
  G_j^{(n)}(t)\right\rvert
\left\lvert  G_j^{(n)}(t)\right\rvert \)\\
&\le C f_n(t) + C\sum_{j=1,2}\left|G_j^{(n-1)}(t)\right|^2,
\end{align*}
for some $C$ independent of $t$ and $n$, since $\nabla^2 V, \nabla^2 K\in
L^\infty$, and where we have used \eqref{eq:dotG1}-\eqref{eq:dotG2} and
Young's inequality. By Gronwall's Lemma, we infer
\begin{equation*}
  f_n(t)\le f_n(0)e^{Ct} +C\int_0^t e^{C(t-s)}f_{n-1}(s)ds.
\end{equation*}
With our definition of the scheme, $f_n(0)$ does not depend on $n$:
\begin{equation*}
  f_n(0)=\sum_{j=1,2} \(\left|\IM \int \bar a_j\nabla a_j\right|^2+
  \left| \int z|a_j(z)|^2dz\right|^2 \)=:C_0. 
\end{equation*}
Therefore,
\begin{equation*}
  f_n(t)\le C_0 e^{Ct} +C\int_0^t e^{C(t-s)}f_{n-1}(s)ds,
\end{equation*}
and by induction, we infer
\begin{equation*}
  f_n(t)\le 2C_0 e^{3Ct},\quad t\ge 0. 
\end{equation*}
By using energy estimates (applying the operators $y$ and $\nabla$
successively to the equations), we infer that there exists $C_1$
independent of $t\ge 0$ and $n$ such that 
\begin{equation*}
  \sum_{j=1,2}\left\|v_j^{(n)}(t)\right\|_{\Sigma^k} \le C_1 e^{C_1 t}. 
\end{equation*}
The convergence of the sequence $(v_1^{(n)},v_2^{(n)})$ then follows:
we check that $v_n$ converges in 
$C([0,T];\Sigma)$ if 
$T>0$ is sufficiently small. To simplify the presentation, we present
this argument in the case of a single envelope equation, the case of
\eqref{eq:vn} bearing no extra difficulty:
\begin{equation}\label{eq:vn2}
i\d_{t}v^{(n)}+\frac{1}{2}\Delta v^{(n)}=\frac{1}{2} \<y, M(t)y\>
v^{(n)}+\<Q(t){G}^{(n-1)}(t), y\>v^{(n)},
\end{equation}
where $Q(t)$ is a real-valued, symmetric matrix, with $Q\in
L^\infty(\R_+)$. Denoting by
\begin{equation*}
  H(t)=-\frac{1}{2}\Delta + \frac{1}{2} \<y, M(t)y\>,
\end{equation*}
we have
\begin{align*}
  i\d_t\( v^{(n)}-v^{(n-1)}\)&=H\( v^{(n)}-v^{(n-1)}\) +
  \<Q(t){G}^{(n-1)}(t), y\>\( v^{(n)}-v^{(n-1)}\)\\
&\phantom{=}  +\<Q(t)\({G}^{(n-1)}(t)-{G}^{(n-2)}(t)\) , y\>v^{(n-1)} 
\end{align*}
Energy estimates and the above uniform bound yield
\begin{align*}
  \left\|v^{(n)}(t)-v^{(n-1)}(t)\right\|_{L^2}&\le
  C\int_0^t\left|{G}^{(n-1)}(s)-{G}^{(n-2)}(s)\right|ds \\
&\le C \int_0^t
  e^{C_1 s} \left\|v^{(n-1)}(s)-v^{(n-2)}(s)\right\|_{\Sigma}ds
\end{align*}
By applying the operators $y$ and $\nabla$ to \eqref{eq:vn2}, we
obtain similarly:
\begin{equation*}
  \left\|v^{(n)}(t)-v^{(n-1)}(t)\right\|_{\Sigma}
\le C \int_0^t
  e^{C_1 s} \left\|v^{(n-1)}(s)-v^{(n-2)}(s)\right\|_{\Sigma}ds.
\end{equation*}
Therefore, we can find $T>0$ such that the sequence $v^{(n)}$
converges in $C([0,T];\Sigma)$, to $v\in C([0,T];\Sigma^k)$. The
uniform bounds for the sequence $v^{(n)}$ imply that $v$ is global in
time: $v\in C(\R_+;\Sigma^k)$, with $\Sigma^k$-norms growing at most
exponentially in time. 
\end{proof}

\section{Convergence in supercritical cases: scheme of the proof}
\label{sec:scheme}

We present the proof of Theorem~\ref{theo:alpha0} in details; the
proof of Theorem~\ref{theo:alpha12} can easily be adapted (see
Remark~\ref{rem:12} below).

\subsection{The general picture}

 In \cite{APPP11,CaCa11}, where the case of only one wave packet is
considered, the
 proof of stability  relies on a change of unknown function:
 writing 
 \begin{equation*}
   \psi^\eps(t,x) = \eps^{-d/4}u^\eps\(t,\frac{x-q(t)}{\sqrt\eps}
   \)e^{i\(S^\eps(t)+(x-q(t))\cdot  p(t)\)/\eps},
 \end{equation*}
with $S^\eps$, $q$ and $p$ as given by the construction of the
approximate solution, it is \emph{equivalent} to work on $\psi^\eps$ or
$u^\eps$ in order to prove an error estimate, since
\begin{equation*}
  \|\psi^\eps(t)-\apsi (t)\|_{L^2(\R^d)} = \|u^\eps(t)-u (t)\|_{L^2(\R^d)}.
\end{equation*}
Passing from the unknown $\psi^\eps$ to $u^\eps$ amounts to using very
fine geometric properties related to the dynamics: the modified action $S^\eps$,
and $(q,p)$. One
changes the origin in phase space, to work in the moving frame
associated to the wave packet. In the case of two wave packets, there
are two moving frames, so the approach that we follow is different. We
construct a solution to \eqref{eq:r3alpha} of the form
\begin{equation}
  \label{eq:psisurcrit}
  \psi^\eps(t,x) = \eps^{-d/4} \sum_{j=1,2}u_j^\eps 
\left(t,\frac{x-q_j(t)}{\sqrt\eps}\right)e^{i\left(S_j^\eps(t)+p_j(t)\cdot
    (x-q_j(t))\right)/\eps},
\end{equation}
where the quantities $(q_j,p_j)$ and $S_j^\eps$ are those given by the
construction of $\apsi$, so we consider two unknown functions,
$u_1^\eps$ and $u_2^\eps$. To do so, we derive formally a system for
$(u_1^\eps,u_2^\eps)$, which is morally equivalent to
\eqref{eq:r3alpha}: rigorously, the solution to this system yields a
solution to \eqref{eq:r3alpha}, and by uniqueness for
\eqref{eq:r3alpha}, the relation \eqref{eq:psisurcrit} is valid. In
turn, the construction of the solution $(u_1^\eps,u_2^\eps)$ on
arbitrary time intervals relies on a bootstrap argument, consisting of
a comparison of a modification of $(u_1^\eps,u_2^\eps)$ with
$(u_1,u_2)$, defined in \eqref{eq:systenv0}. This modification
eventually corresponds to the presence 
of the phase shifts $\theta_j$ in Theorem~\ref{theo:alpha0}. 
\smallbreak

In order to shorten the formulas, we consider indices in $\Z/2\Z$:
typically, $q_j$ stands for $q_1$ whenever $j=1$ or $3$. Plugging
\eqref{eq:psisurcrit} into \eqref{eq:r3alpha} in the case $\alpha=0$,
we find:
\begin{equation*}
  i\eps \d_t \psi^\eps + \frac{\eps^2}{2}\Delta \psi^\eps -
  V\psi^\eps -\(K\ast 
  |\psi^\eps|^2\)
  \psi^\eps=\eps^{-d/4}\sum_{j=1,2}e^{i\phi_j^\eps(t,x)} N_j^\eps,
\end{equation*}
where we have denoted
\begin{align*}
  & \phi_j^\eps(t,x) = S_j^\eps(t)+p_j(t)\cdot    \(x-q_j(t)\),\\
& N_j^\eps = i\eps \d_t 
  u_j^\eps -
  u_j^\eps\d_t \phi_j^\eps + \frac{\eps}{2}\Delta u_j^\eps
  -\frac{|p_j(t)|^2}{2}u_j^\eps 
-V\(t,q_j(t)+y_j\sqrt\eps\) u_j^\eps - \tilde V^{\rm
    NL}_j u_j^\eps, \\
& \tilde V_j^{\rm NL}(t,y_j) = \int K\(\sqrt\eps
  (y_j-z)\)|u_j^\eps(t,z)|^2dz \\
&\phantom{\tilde V_j^{\rm NL}(t,y_j) =}+ \int K\(q_j-q_{j+1} + \sqrt\eps
  (y_j-z)\)|u_{j+1}^\eps(t,z)|^2dz\\
&\phantom{\tilde V_j^{\rm NL}(t,y_j) =}+2\RE
e^{i\(S_j^\eps-S_{j+1}^\eps-q_j\cdot p_j + q_{j+1}\cdot 
  p_{j+1}\)/\eps}\times\\
&\phantom{\tilde V_j^{\rm NL}}\times\int
K\(\sqrt\eps(y_j-z)\)e^{iz\cdot (p_j-p_{j+1})/\sqrt\eps} 
u_j^\eps(t,z)\overline u_{j+1}^\eps\( t,z
+\frac{q_j-q_{j+1}}{\sqrt\eps}\) dz .
\end{align*}
Note that the computations which we do not detail correspond to the
computations presented in Section~\ref{sec:rough}, up to the fact that
now, we do not perform Taylor expansions for $V$ or $K$. As in
Section~\ref{sec:rough}, we distinguish the variables $y_1$ and $y_2$.
\smallbreak

Our approach
consists in considering the set of coupled, nonlinear equations
$$N_1^\eps=N_2^\eps=0.$$
It is important to notice at this stage of the
construction that this system conserves formally the $L^2$ norms:
since we naturally impose $u^\eps_{j\mid t=0}=a_j$, we have 
\begin{equation*}
  \|u^\eps_j(t)\|_{L^2(\R^d)} = \|a_j\|_{L^2(\R^d)} ,\quad j=1,2,
\end{equation*}
as long as $(u_1^\eps,u_2^\eps)$ is well defined. This property is the
reason why we can perform important reductions in the system. Taking into
account the expression of the modified actions $S_j^\eps$, we
find:  
\begin{align*}
  N_j^\eps &= i\eps \d_t 
  u_j^\eps -
  u_j^\eps\(\dot S_j^\eps(t) +\sqrt \eps \dot p_j(t)\cdot
  y_j-p_j(t)\cdot \dot q_j(t)\) \\
& + \frac{\eps}{2}\Delta u_j^\eps
  -\frac{|p_j(t)|^2}{2}u_j^\eps 
-V\(t,q_j(t)+y_j\sqrt\eps\) u_j^\eps - \tilde V^{\rm
    NL}_j(t,y_j) u_j^\eps\\
&= i\eps \d_t 
  u_j^\eps- \(\sqrt \eps \dot p_j(t)\cdot
  y_j-p_j(t)\cdot \dot q_j(t)\)u_j^\eps \\
&- \(\frac{1}{2}|p_j|^2 -V(t,q_j) -
  K(0)\|a_j\|_{L^2} - K(q_j-q_{j+1})\|a_{j+1}\|_{L^2}^2\)u_j^\eps\\
& + \(\sqrt\eps
  \nabla K(0)\cdot G_j(t) +\sqrt\eps \nabla K(q_j-q_{j+1})\cdot
  G_{j+1}(t)\)u_j^\eps \\ 
& + \frac{\eps}{2}\Delta u_j^\eps
  -\frac{|p_j(t)|^2}{2}u_j^\eps 
-V\(t,q_j(t)+y_j\sqrt\eps\) u_j^\eps - \tilde V^{\rm
    NL}_j(t,y_j) u_j^\eps\\
&=i\eps \d_t  u_j^\eps+ \frac{\eps}{2}\Delta u_j^\eps-\sqrt \eps \dot p_j(t)\cdot
  y_j u_j^\eps\\
&-\(V\(t,q_j(t)+y_j\sqrt\eps\) -V(t,q_j)\) u_j^\eps\\
&- \(\tilde V^{\rm
    NL}_j(t,y_j) -
  K(0)\|a_j\|_{L^2} - K(q_j-q_{j+1})\|a_{j+1}\|_{L^2}^2\)u_j^\eps\\
& + \(\sqrt\eps
  \nabla K(0)\cdot G_j(t) +\sqrt\eps \nabla K(q_j-q_{j+1})\cdot
  G_{j+1}(t)\)u_j^\eps .
\end{align*}
If we now take into account the expression of $\dot p_j$, given in
\eqref{eq:trajmodif}, we infer: 
\begin{align*}
  N_j^\eps &=i\eps \d_t  u_j^\eps+ \frac{\eps}{2}\Delta u_j^\eps\\
&+ \sqrt\eps y_j\cdot \(\nabla V(t,q_j)+ \|a_j\|_{L^2}^2\nabla K(0) + 
\|a_{j+1}\|_{L^2}^2\nabla K(q_j-q_{j+1})\)u_j^\eps\\ 
&-\(V\(t,q_j(t)+y_j\sqrt\eps\) -V(t,q_j)\) u_j^\eps\\
&- \(\tilde V^{\rm
    NL}_j(t,y_j) -
  K(0)\|a_j\|_{L^2}^2 - K(q_j-q_{j+1})\|a_{j+1}\|_{L^2}^2\)u_j^\eps\\
& + \(\sqrt\eps
  \nabla K(0)\cdot G_j(t) +\sqrt\eps \nabla K(q_j-q_{j+1})\cdot
  G_{j+1}(t)\)u_j^\eps .
\end{align*}
It is now natural to introduce the following notations:
\begin{align*}
  V_j^\eps(t,y_j) & = \frac{1}{\eps}\(V\(t,q_j(t)+y_j\sqrt\eps\)
  -V\(t,q_j(t)\)  - \sqrt\eps y_j\cdot \nabla V\(t,q_j(t)\)\),\\
K_{j,{\rm diag}}^\eps(t,y_j)&= \frac{1}{\eps}\(K\(\sqrt\eps y_j\)- K(0)-\sqrt\eps
y_j\cdot \nabla K(0)\),\\
K_{j,{\rm off}}^\eps(t,y_j)&=\frac{1}{\eps}\(K\(q_j-q_{j+1}+\sqrt\eps y_j\)-
K(q_j-q_{j+1})-\sqrt\eps 
y_j\cdot \nabla K(q_j-q_{j+1})\) . 
\end{align*}
From the assumptions on $V$ and $K$, there exists $C>0$
independent of $\eps\in (0,1]$ such that
\begin{equation}
  \label{eq:borneeps}
  \sum_{2\le |\alpha|\le 6}\left\|
    \d^\alpha_{y_j}V_j^\eps(t)\right\|_{L^\infty} +\sum_{2\le |\alpha|\le 6}\left\|
    \d^\alpha_{y_j}K_{j}^\eps(t)\right\|_{L^\infty}\le C,
\end{equation}
where $K_j^\eps$ stands for $K_{j,{\rm diag}}^\eps$ or $K_{j,{\rm
    off}}^\eps$, indistinctly. 
In view of Taylor's formula, we have:
\begin{align}
  V_j^\eps(t,y_j) &= \int_0^1 \< y_j,\nabla^2 V\(t,q_j(t)+ \theta
  y_j\sqrt \eps\)y_j\>(1-\theta)d\theta, \label{eq:Veps}\\
 K_{j,{\rm diag}}^\eps(t,y_j)&= \int_0^1 \< y_j,\nabla^2 K\( \theta
  y_j\sqrt \eps\)y_j\>(1-\theta)d\theta \label{eq:Kepsdiag},\\
K_{j,{\rm off}}^\eps(t,y_j)&= \int_0^1 \< y_j,\nabla^2 K\(q_j(t)-q_{j+1}(t)+ \theta
  y_j\sqrt \eps\)y_j\>(1-\theta)d\theta . \label{eq:Kepsoff}
\end{align}
Therefore, we consider the coupled system (coupling is present
through $K_j^\eps$):
\begin{equation}
  \label{eq:ujeps}
\left\{
  \begin{aligned}
  i\d_t u_j^\eps+\frac{1}{2}\Delta u_j^\eps &= V_j^\eps (t,y_j)u^\eps_j
  +\(K_{j,{\rm diag}}^\eps\ast |u_j^\eps|^2\)u_j^\eps +\(K_{j,{\rm
      off}}^\eps\ast |u_{j+1}^\eps|^2\)u_j^\eps  \\
 -\frac{1}{\sqrt\eps}& \nabla K(0)\cdot \(\int
z\(|u_j^\eps(t,z)|^2-|u_j(t,z)|^2\)dz \) u_j^\eps \\
 -\frac{1}{\sqrt\eps} &\nabla K(q_j-q_{j+1})\cdot \(\int
z\(|u_{j+1}^\eps(t,z)|^2-|u_{j+1}(t,z)|^2\)dz \) u_j^\eps\\
 + \frac{1}{\eps}&\(2\RE W_j^\eps(t,y_j)\)u_j^\eps,
\end{aligned}
\right.
\end{equation}
with  
\begin{align*}
  W_j^\eps(t,y_j) &= e^{i\(S_j^\eps-S_{j+1}^\eps-q_j\cdot p_j + q_{j+1}\cdot
  p_{j+1}\)/\eps}\times\\
&\times\int K\(\sqrt\eps(y_j-z)\)e^{iz\cdot (p_j-p_{j+1})/\sqrt\eps}
u_j^\eps(t,z)\overline u_{j+1}^\eps\( t,z
+\frac{q_j-q_{j+1}}{\sqrt\eps}\) dz. 
\end{align*}

\subsection{Further simplification and bootstrap argument}
The last three  terms in \eqref{eq:ujeps} are singular in the limit
$\eps\to 0$. However, the singularity of the last term is expected to
be artificial, since in view of Proposition~\ref{prop:microloc}, it
should even be small, provided we have uniform estimates for
$u_j^\eps$ in $\Sigma^3$. The other two singular terms have an
interesting feature: they are real-valued, and depend only on time, so
we can treat them thanks to a gauge transform. Introduce
\begin{equation}\label{eq:exactesimplif}
  \begin{aligned}
    i\d_t \tilde u_j^\eps+\frac{1}{2}\Delta \tilde u_j^\eps &= V_j^\eps
  (t,y_j)\tilde u^\eps_j 
  +\(K_{j,{\rm diag}}^\eps\ast |\tilde u_j^\eps|^2\)\tilde u_j^\eps
  +\(K_{j,{\rm off}}^\eps\ast
  |\tilde u_{j+1}^\eps|^2\)\tilde u_j^\eps\\
&+\frac{1}{\eps}\(2\RE
  \tilde W_j^\eps\)\tilde u_j^\eps  ,
  \end{aligned}
  \end{equation}
with initial data $\tilde
  u_{j\mid t=0}^\eps=a_j$, and where we have denoted
\begin{align*}
 & \tilde W_j^\eps =e^{i\(S_j^\eps-S_{j+1}^\eps-q_j\cdot p_j + q_{j+1}\cdot
  p_{j+1}\)/\eps}e^{i(\theta_j^\eps-\theta_{j+1}^\eps)}\times\\
& \times\int K\(\sqrt\eps(y_j-z)\)e^{iz\cdot (p_j-p_{j+1})/\sqrt\eps}
\tilde u_j^\eps(t,z)\overline{\tilde u}_{j+1}^\eps\( t,z
+\frac{q_j-q_{j+1}}{\sqrt\eps}\) dz ,
\end{align*}
with
\begin{align*}
\theta^\eps_j(t)&=  \frac{1}{\sqrt\eps}\int_0^t \nabla K(0)\cdot \(\int
z\(|\tilde u_j^\eps(s,z)|^2-|u_j(s,z)|^2\)dz\)ds\\
&+\frac{1}{\sqrt\eps}\int_0^t\nabla
K(q_j(s)-q_{j+1}(s))\cdot \(\int 
z\(|\tilde u_{j+1}^\eps(s,z)|^2-| u_{j+1}(s,z)|^2\)dz\)ds.
\end{align*}
We then have: $u_j^\eps(t,y) = \tilde
u_j^\eps(t,y)e^{i\theta^\eps_j(t)}$. 
Note that $|u_j^\eps|= |\tilde u_j^\eps|$, so it is equivalent to pass
from $u_j^\eps$ to $\tilde u_j^\eps$,  or from $\tilde u_j^\eps$ to
$u_j^\eps$. 
In view of these reductions, \emph{in a first approximation},
Theorem~\ref{theo:alpha0} 
stems from:
\begin{theorem}\label{theo:final}
  Let $d\ge 1$ and $a_1,a_2\in \Sigma^6$. Assume 
  that $V$ and $K$ are real-valued and:
 \begin{align*}
 & V\in C^6(\R_{+}\times \R^d;\R),\quad \text{and}\quad   \d_x^\beta V\in
 L^\infty\(\R_{+}\times\R^d\),\quad  2\le |\beta|\le 6.\\
 &K\in W^{6,\infty}(\R^d;\R).
 \end{align*}
Let $T>0$. There exists $\eps_0>0$ such that for
$\eps\in(0,\eps_0]$, \eqref{eq:exactesimplif} has a unique
solution $(\tilde u_1^\eps,\tilde u_2^\eps)\in
C([0,T];\Sigma^3)^2$. Moreover, there exists $C$ independent of
$\eps\in (0,\eps_0]$ such that
\begin{equation}\label{eq:restefinal}
  \sup_{t\in [0,T]}\left\| \tilde u_1^\eps
    (t)-u_1(t)\right\|_{\Sigma^3}+\sup_{t\in [0,T]}\left\| \tilde u_2^\eps
    (t)-u_2(t)\right\|_{\Sigma^3} \le C\sqrt\eps. 
\end{equation}
\end{theorem}
Several comments are in order. First, this result implies that for
$j=1,2$, 
$\dot \theta_j^\eps$ is bounded on $[0,T]$, uniformly in $\eps$. To
get the result stated in  Theorem~\ref{theo:alpha0}, we will prove
that the functions $\theta_j^\eps$ converge as $\eps\to 0$, by
performing a second order asymptotic expansion of
$(\tilde u_1^\eps,\tilde u_2^\eps)$ (Theorem~\ref{theo:final} yields
the first order asymptotic expansion). 

Even in the case of a single wave packet, this result is new, since we
do not assume $\nabla K(0)=0$. In that case, the last term in
\eqref{eq:exactesimplif} vanishes, and the proof that we present below
becomes simpler. 

The proof is based on a bootstrap argument detailed
  in Section~\ref{sec:bootstrap}. For fixed $\eps>0$, 
\eqref{eq:exactesimplif} has a unique, local solution: $(\tilde
u_1^\eps,\tilde u_2^\eps)\in 
C([0,\tau^\eps];\Sigma^3)^2$, for some $\tau^\eps>0$. This can be
proven by adapting the approach presented in
Section~\ref{sec:refined}. To prove the theorem, we use energy
estimates to prove that so long as $(\tilde u_1^\eps,\tilde u_2^\eps)$
is bounded in $C([0,\tau];\Sigma^3)^2$, $\tau^\eps\le \tau\le T$,
\eqref{eq:restefinal} is true. Therefore, choosing $\eps_0>0$
sufficiently small, $(\tilde u_1^\eps,\tilde u_2^\eps)\in
C([0,T];\Sigma^3)^2$ for $\eps\in (0,\eps_0]$, and
\eqref{eq:restefinal} is satisfied. 

The reason why we work in $\Sigma^3$ and not in a larger space is that
we want to be able to neglect $\tilde W_j^\eps$: because of the
singular factor $1/\eps$ in front of the last term in
\eqref{eq:exactesimplif}, we need to prove $\tilde W_j^\eps=o(\eps)$,
and Proposition~\ref{prop:microloc} suggests that
we need to work in $\Sigma^3$, in which case 
$\tilde W_j^\eps=\O(\eps^{3/2})$. To differentiate $V_j^\eps$ and
$K_j^\eps$ three times (we work in $\Sigma^3$),
\eqref{eq:Veps}--\eqref{eq:Kepsoff} and
Proposition~\ref{prop:microloc} suggest to work with  
the regularity stated in Theorem~\ref{theo:final} (the same as in
Theorem~\ref{theo:alpha0}). 

\begin{remark}\label{rem:12}
  In the case $\alpha=1/2$, one can consider that all the terms
  involving $K$ are multiplied by $\sqrt\eps$. As a first consequence,
  it is enough to work in $\Sigma^2$ to prove that the term involving $\tilde
  W_j^\eps$ is negligible. By working in $\Sigma^2$, we only need to
  differentiate $V_j^\eps$ and
$K_j^\eps$ twice, hence the regularity assumption in
Theorem~\ref{theo:alpha12}. Finally, since the phase shift relating
$\tilde u_j^\eps$ and $u_j^\eps$ is multiplied by $\sqrt\eps$, it is
$\O(\sqrt\eps)$, as opposed to $\O(1)$ in the case $\alpha=0$. 
\end{remark}

\section{The bootstrap argument}
\label{sec:bootstrap}

In this section, we prove Theorem~\ref{theo:final}. More precisely, we
focus on \eqref{eq:restefinal}, in view of the discussion at the end
of Section~\ref{sec:scheme}.
\smallbreak

In Section~\ref{sec:scheme}, we have essentially resumed the
computations of Section~\ref{sec:rough}, up to two aspects:
\begin{itemize}
\item We have not used Taylor's formula for $V$ and $K$. 
\item The terms $\theta_j^\eps$ do not appear in the case of the
  $u_j$'s (replacing $u_j^\eps$ with $u_j$ in the expression of
  $\theta_j^\eps$ yields $\theta_j^\eps=0$).
\end{itemize}
In Section~\ref{sec:mixed}, we have seen that the analogue of the term
$W_j^\eps$ (or, equivalently, $\tilde W_j^\eps$) is negligible in the
limit $\eps\to 0$. These properties can be summarized as follows: the
functions $u_1$ and $u_2$ solve
\begin{equation}
  \label{eq:uj2}
    i\d_t u_j+\frac{1}{2}\Delta  u_j = V_j^\eps u_j
  +\(K_{j,{\rm diag}}^\eps\ast |u_j|^2\) u_j
  +\(K_{j,{\rm off}}^\eps\ast
  |u_{j+1}|^2\) u_j+\rho_j^\eps,
\end{equation}
where $\rho_j^\eps$ is given by the formula:
\begin{align*}
  \rho_j^\eps &= \(V_j^0-V_j^\eps\)u_j + \(\(K_{j,{\rm diag}}^\eps -
  K_{j,{\rm diag}}^0\)\ast |u_j|^2\)u_j\\
&\quad + \(\(K_{j,{\rm off}}^\eps -
  K_{j,{\rm off}}^0\)\ast |u_{j+1}|^2\)u_j,
\end{align*}
where $V_j^0$, $K_j^0$ are given by
\eqref{eq:Veps}--\eqref{eq:Kepsoff} with
$\eps=0$. 
We infer from \eqref{eq:Veps}--\eqref{eq:Kepsoff} and 
Proposition~\ref{prop:existenv0} 
that for all $T>0$, there exists $C>0$ independent of $\eps\in (0,1]$
such that
\begin{equation*}
  \sup_{t\in [0,T]}\|\rho_j^\eps(t)\|_{\Sigma^3}\le C\sqrt\eps. 
\end{equation*}
Since the bootstrap argument runs in $\Sigma^3$, it is natural to work
with such an estimate for the source term. This in turn imposes to
work with $a_j\in \Sigma^6$, as well as $V$ and $K$ as in
Theorem~\ref{theo:final}.  
\smallbreak

Set $w^\eps_j = \tilde u_j^\eps -u_j$: subtracting \eqref{eq:uj2} from
\eqref{eq:exactesimplif}, we see that the error satisfies the coupled
system, for $j=1,2$,
\begin{equation}
  \label{eq:w}
  \left\{
    \begin{aligned}
      i\d_t w_j^\eps + \frac{1}{2}\Delta w_j^\eps &=V_j^\eps w_j^\eps
      +\(K_{j,{\rm diag}}^\eps\ast |\tilde u_j^\eps|^2\)\tilde
      u_j^\eps 
  -\(K_{j,{\rm diag}}^\eps\ast |u_j|^2\)
      u_j\\
& +\(K_{j,{\rm off}}^\eps\ast
  |\tilde u_{j+1}^\eps|^2\)\tilde u_j^\eps-
\(K_{j,{\rm off}}^\eps\ast
  |u_{j+1}|^2\) u_j\\
&+\frac{1}{\eps}\(2\RE
  \tilde W_j^\eps\)\tilde u_j^\eps -\rho_j^\eps ,
    \end{aligned}
\right.
\end{equation}
with initial data $w^\eps_{j\mid t=0}=0$. Fix $T>0$ once and
for all in the course of the proof. By Proposition~\ref{prop:existenv0}, 
there exists $C_0>0$ such that
\begin{equation*}
  \sup_{t\in [0,T]}\|u_1(t)\|_{\Sigma^3} + \sup_{t\in
    [0,T]}\|u_2(t)\|_{\Sigma^3}\le C_0. 
\end{equation*}
Since $w_{j \mid t=0}^\eps=0$ and $\tilde u_j^\eps \in
C([0,\tau^\eps];\Sigma^3)$ for some $\tau^\eps$, we can find
$t^\eps>0$ such that
\begin{equation}
  \label{eq:solong}
  \|w_1^\eps(t)\|_{\Sigma^3} + \|w_2^\eps(t)\|_{\Sigma^3}\le C_0
\end{equation}
for $0\le t\le t^\eps$. So long as \eqref{eq:solong} holds, we perform
energy estimates, to show that \eqref{eq:restefinal} is true, with a
constant $C$ independent of $\eps$. It will follow that up to choosing
$\eps\in (0,\eps_0]$ with $\eps_0>0$ sufficiently small,
\eqref{eq:solong} holds for $t\in [0,T]$, which yields
Theorem~\ref{theo:final}. 
\smallbreak

\begin{notation}
 For two positive numbers $a^\eps$ and $b^\eps$,
the notation $ a^\eps\lesssim b^\eps$
means that there exists $C>0$ \emph{independent of} $\eps$ such that
for all $\eps\in (0,1]$, $a^\eps\le Cb^\eps$.  
\end{notation}
 
Note that so long as \eqref{eq:solong} holds, similarly to the case of
$I_j^\eps$, 
Proposition~\ref{prop:microloc} implies
\begin{equation*}
  \left\|\tilde W_j^\eps(t)\right\|_{W^{3,\infty}}\lesssim \eps^{3/2},
\end{equation*}
hence 
\begin{equation*}
  \left\|\frac{1}{\eps}\(2\RE
  \tilde W_j^\eps\)\tilde u_j^\eps\right\|_{\Sigma^3}\lesssim \sqrt\eps. 
\end{equation*}
Therefore, the last line in \eqref{eq:w}, viewed as a source
term, is $\O(\sqrt\eps)$ in $\Sigma^3$, so long as \eqref{eq:solong}
holds. The other terms in \eqref{eq:w} can then be considered
as linear terms, in view of the application of Gronwall's Lemma. 
\smallbreak

We write
\begin{align*}
  \(K_{j,{\rm diag}}^\eps\ast |\tilde u_j^\eps|^2\)\tilde
      u_j^\eps 
  -\(K_{j,{\rm diag}}^\eps\ast |u_j|^2\)
      u_j &=\(K_{j,{\rm diag}}^\eps\ast |\tilde u_j^\eps|^2\)
      w_j^\eps \\
&+
  \(K_{j,{\rm diag}}^\eps\ast \(|\tilde u_j^\eps|^2-|u_j|^2\)\)u_j\\
 &=\(K_{j,{\rm diag}}^\eps\ast |\tilde u_j^\eps|^2\)
      w_j^\eps \\
&+
  \(K_{j,{\rm diag}}^\eps\ast \(|w_j^\eps|^2 + 2\RE \bar u_j w_j^\eps\)\)u_j,
\end{align*}
and a similar relation for the off-diagonal kernel. We develop the
general convolution, where $K^\eps$ is of the form \eqref{eq:Kepsdiag}:
\begin{align*}
  \(K^\eps \ast f\)g = 
&\(\iint_0^1 \< y-z,\nabla^2 K\( \theta
  (y-z)\sqrt \eps\)(y-z)\>(1-\theta)d\theta f(z)dz \) g \\ 
= & \<y, \(\iint_0^1 (1-\theta)\nabla^2 {K}\( \theta
  (y-z)\sqrt \eps\) d\theta f(z)dz\) y \>g\\
&+\( \iint_0^1\<z,\nabla^2 {K}\( \theta
  (y-z)\sqrt \eps\)z\>(1-\theta)d\theta f(z)dz \) g\\
& - 2\<y, \iint_0^1 (1-\theta)\nabla^2 {K}\( \theta
  (y-z)\sqrt \eps\)z d\theta f(z)dz\>g.
\end{align*}
The same computation is available for \eqref{eq:Kepsoff}, with heavier
notations, so we leave it out. From this we readily infer
\begin{align*}
  \left\|\(K_{j,{\rm diag}}^\eps\ast \(|w_j^\eps|^2 + 2\RE \bar u_j
    w_j^\eps\)\)u_j \right\|_{\Sigma^3}\lesssim
  \|w_j^\eps\|_{\Sigma^3}^2 +  \|w_j^\eps\|_{\Sigma^3}\lesssim
  \|w_j^\eps\|_{\Sigma^3} 
\end{align*}
where we have used Proposition~\ref{prop:existenv0},
Cauchy--Schwarz inequality, and \eqref{eq:solong} for the last
estimate. We can infer an $L^2$ estimate for $w_j^\eps$: since all the
terms of the form $K_j^\eps \ast |\tilde u^\eps|^2 $ are real valued,
the standard energy estimate yields
\begin{align*}
  \|w_j^\eps(t)\|_{L^2}&\le \int_0^t \left\|\(K_{j,{\rm diag}}^\eps\ast
    \(|w_j^\eps|^2 + 2\RE \bar u_j 
    w_j^\eps\)\)u_j \right\|_{L^2}ds \\
&+ \int_0^t \left\|\(K_{j,{\rm off}}^\eps\ast
    \(|w_{j+1}^\eps|^2 + 2\RE \bar u_{j+1} 
    w_{j+1}^\eps\)\)u_j \right\|_{L^2}ds\\
&+  \int_0^t\left\|\frac{1}{\eps}\(2\RE
  \tilde W_j^\eps\)\tilde u_j^\eps\right\|_{L^2}ds+
\int_0^t\|\rho_j^\eps(s)\|_{L^2} ds\\
&\lesssim \int_0^t\(\|w_j^\eps(s)\|_{\Sigma^3} +
\|w_{j+1}^\eps(s)\|_{\Sigma^3}\)ds +\int_0^t \sqrt\eps ds.
\end{align*}
To pass from this $L^2$ estimate to a $\Sigma^3$ estimate, we have to
assess the action of the operators of multiplication by $y_j$ and
$\nabla_{y_j}$ on \eqref{eq:w}. First, $\nabla_{y_j}$ commutes with
the left hand side of \eqref{eq:w}, but not with the right hand
side. We write
\begin{equation*}
  \d_{k\ell m}^3 \( V_j^\eps w_j^\eps\) = 
  V_j^\eps \d_{k\ell m}^3w_j^\eps+ \sum_{{0\le |\alpha|\le 2}\atop
{|\beta|= 3-|\alpha|}} c(\alpha,k,\ell,m) \d^\beta V_j^\eps \d^\alpha
w_j^\eps. 
\end{equation*}
The first term vanishes in an $L^2$ estimate of $\d_{k\ell
  m}^3w_j^\eps$, and in view of \eqref{eq:borneeps}, for all multi-indices
$\alpha,\beta$ with  $0\le
|\alpha|\le 2$ and $|\beta|= 3-|\alpha|$, 
\begin{equation*}
  \left\|  \d^\beta V_j^\eps \d^\alpha
w_j^\eps\right\|_{L^2}\lesssim \|w_j^\eps\|_{\Sigma^3}. 
\end{equation*}
\begin{remark}
The presence of the potential $V_j^\eps$, which is morally a
time dependent harmonic potential, forces us to work in $\Sigma^3$,
and not simply in $H^3(\R^d)$: this is a standard feature of such
potentials, whose associated dynamics consists of rotations in phase
space, so the regularity/decay of the functions 
must be the same in space and in frequency. 
\end{remark}
The terms $\(K_{j,{\rm diag}}^\eps\ast |\tilde u_j^\eps|^2\)
      w_j^\eps$ and $\(K_{j,{\rm off}}^\eps\ast |\tilde u_{j+1}^\eps|^2\)
      w_j^\eps$ are treated similarly, and produce a term of the form
      real$\times \d_{k\ell m}^3w_j^\eps$, plus a term controlled in
      $L^2$ by $\|w_j^\eps\|_{\Sigma^3}$, so long as \eqref{eq:solong}
      holds. 
\smallbreak

On the other hand, the multiplication by $y_j$
commutes with the right hand side of \eqref{eq:w}, but not with the
left hand side: 
\begin{equation*}
  \left[ i\d_t+\frac{1}{2}\Delta,y\right]= \nabla,
\end{equation*}
so the commutation errors for the equation satisfied by $|y_j|^3
w_j^\eps$ consists of a linear combination, with constant coefficients,
of terms of the form $y_j^\alpha\d^\beta w_j^\eps$,  with
$|\alpha|+|\beta|=3$. We end up with, so long as \eqref{eq:solong}
holds: 
\begin{equation*}
  \|w_j^\eps(t)\|_{\Sigma^3}\lesssim \int_0^t\(\|w_j^\eps(s)\|_{\Sigma^3} +
\|w_{j+1}^\eps(s)\|_{\Sigma^3}\)ds +\int_0^t \sqrt\eps ds.
\end{equation*}
Gronwall's Lemma yields \eqref{eq:restefinal}, hence
Theorem~\ref{theo:final}. 

\section{Second order expansion and limiting phase shifts}
\label{sec:DA2}

In view of Theorem~\ref{theo:final}, the phase shifts $\theta_j^\eps$
are such that $\dot \theta_j^\eps$ are bounded on $[0,T]$, uniformly
in $\eps\in (0,\eps_0]$, since $|u_j^\eps|^2=|\tilde
u_j^\eps|^2=|u_j|^2 + \O(\sqrt\eps)$. To study the limit of
$\theta_j^\eps$ as $\eps \to 0$, we need to perform a second order
expansion of $(\tilde u_1^\eps,\tilde u_2^\eps)$ as $\eps \to 0$, to
understand the contribution of order $\sqrt\eps$. Therefore, we seek
\begin{equation}\label{eq:DA2}
  \tilde u_j^\eps = u_j +\sqrt\eps u_j^{(1)} +\O(\eps). 
\end{equation}
\begin{remark}
  An error term of order $\O(\eps)$ is natural, since one could actually
seek a more general asymptotic expansion to arbitrary order, of the
form
\begin{equation}\label{eq:DAgeneral}
  \tilde u_j^\eps = u_j +\sum_{\ell=1}^N\eps^{\ell/2} u_j^{(\ell)}
  +\O\(\eps^{(N+1)/2}\).  
\end{equation}
\end{remark}
Resuming the arguments presented in Section~\ref{sec:scheme}, we see
that formally, the last line in \eqref{eq:ujeps} is $\O(\eps^\infty)$,
if we work with an infinite regularity. To get a second order
approximation of $\tilde u_j^\eps$, we simply need to prove that this
term is $\O(\eps)$, but we can certainly not perform the study with
only an $\O(\sqrt\eps)$ information like we did in order to establish
Theorem~\ref{theo:final}. 
To compute the limit of $\theta_j^\eps$, we need to
establish the asymptotic behavior of $\tilde u_j^\eps$ up to
$\O(\eps)$ in $\Sigma$, and not only in $L^2$, so we make an extra
regularity assumption. 
We  remark that if in
Theorem~\ref{theo:final}, we require $a_1,a_2\in \Sigma^7$, with
 \begin{align*}
 & V\in C^7(\R_{+}\times \R^d;\R),\quad \text{and}\quad   \d_x^\beta V\in
 L^\infty\(\R_{+}\times\R^d\),\quad  2\le |\beta|\le 7,\\
 &K\in W^{7,\infty}(\R^d;\R),
 \end{align*}
then the conclusions of Theorem~\ref{theo:final} can be replaced by:
$(\tilde u_1^\eps,\tilde u_2^\eps)\in C([0,T];\Sigma^4)^2$ and
\begin{equation*}
  \sup_{t\in [0,T]}\left\| \tilde u_1^\eps
    (t)-u_1(t)\right\|_{\Sigma^4}+\sup_{t\in [0,T]}\left\| \tilde u_2^\eps
    (t)-u_2(t)\right\|_{\Sigma^4} \le C\sqrt\eps. 
\end{equation*}
In particular, $(\tilde u_1^\eps,\tilde u_2^\eps)\in
C([0,T];\Sigma^4)^2$ \emph{uniformly} for $\eps\in (0,\eps_0]$.
Thanks to Proposition~\ref{prop:microloc}, this
enables us to claim that in \eqref{eq:exactesimplif}, 
\begin{equation*}
  \left\|\tilde W_j^\eps(t)\right\|_{W^{3,\infty}}\lesssim \eps^{2},
\end{equation*}
hence 
\begin{equation*}
  \left\|\frac{1}{\eps}\(2\RE
  \tilde W_j^\eps\)\tilde u_j^\eps\right\|_{\Sigma^3}\lesssim \eps. 
\end{equation*} 
To derive an equation for the corrector $u_j^{(1)}$, we
plug \eqref{eq:DA2} into \eqref{eq:exactesimplif}, and discard all the
terms which are, at least formally, $\O(\eps)$, including thus the
last line. The term corresponding to the power $\sqrt\eps$ yields:
\begin{equation}
  \label{eq:reste2}
 \begin{aligned}
  i\d_t u_j^{(1)}+\frac{1}{2}&\Delta u_j^{(1)} = V_j^0 u_j^{(1)}
  +\(K^0_{j,{\rm diag}}\ast|u_j|^2\) u_j^{(1)} + \(K^0_{j,{\rm
      off}}\ast|u_{j+1}|^2\) u_j^{(1)}  \\
& +2\(K^0_{j,{\rm diag}}\ast \RE\(\overline u_j u_j^{(1)} \)\)
u_j + 2\(K^0_{j,{\rm   off}}\ast\RE\(\overline u_{j+1} u_{j+1}^{(1)}\) \)u_j\\
& +{\mathcal V}_j u_j + \({\mathcal K}_{j,{\rm
    diag}}\ast|u_j|^2\) u_j + \({\mathcal K}_{j,{\rm
      off}}\ast|u_{j+1}|^2\) u_j,
\end{aligned} 
\end{equation}
with Cauchy data $ u_{j\mid t=0}^{(1)}=0$,
where we have denoted the third order Taylor expansions
\begin{align*}
  {\mathcal V}_j(t,y)&=\frac{1}{6}\nabla^3 V\(t,q_j(t)\)y\cdot y\cdot y,\\
 {\mathcal K}_{j,{\rm diag}}(y) &=\frac{1}{6}\nabla^3 K(0)y\cdot y\cdot y, \\
{\mathcal K}_{j,{\rm off}} (t,y)&= \frac{1}{6}\nabla^3
K\(q_j(t)-q_{j+1}(t)\)y\cdot y\cdot y .
\end{align*}
These equations are naturally linear in the unknown
$(u_1^{(1)},u_2^{(1)})$.
In view of Proposition~\ref{prop:existenv0}, the last line in
\eqref{eq:reste2}, which corresponds to a source term, belongs to
$C(\R_+;\Sigma^4)$. This non-trivial source term makes $u_j^{(1)}$
non-zero. Even though \eqref{eq:reste2} is a linear system,
it seems easier to prove that it has a unique solution, by proceeding
in the same way as in the proof of Proposition~\ref{prop:existenv0}
(see Section~\ref{sec:refined}). We have:
\begin{proposition}\label{prop:reste2}
Suppose that $a_1,a_2\in \Sigma^7$, and   
\begin{align*}
 & V\in C^7(\R_{+}\times \R^d;\R),\quad \text{and}\quad   \d_x^\beta V\in
 L^\infty\(\R_{+}\times\R^d\),\quad  2\le |\beta|\le 7,\\
 &K\in W^{7,\infty}(\R^d;\R).
 \end{align*}
Then \eqref{eq:reste2} has a unique solution
$  \(u_1^{(1)},u_2^{(1)}\)\in C(\R_+;\Sigma^4)$.
\end{proposition}
Denote by $v_j^\eps = u_j+\sqrt\eps u_j^{(1)}$ the second order
approximate solution, and by $\tilde w_j^\eps = \tilde u_j^\eps -
v_j^\eps$ the corresponding error term. It satisfies $\tilde w_{j\mid
  t=0}^\eps=0$, and
\begin{equation*}
  \left\{
    \begin{aligned}
      i\d_t \tilde w_j^\eps + \frac{1}{2}\Delta \tilde w_j^\eps
      &=V_j^\eps \tilde w_j^\eps
      +\(K_{j,{\rm diag}}^\eps\ast |\tilde u_j^\eps|^2\)\tilde
      u_j^\eps 
  -\(K_{j,{\rm diag}}^\eps\ast |v_j^\eps|^2\)
      v^\eps_j\\
& +\(K_{j,{\rm off}}^\eps\ast
  |\tilde u_{j+1}^\eps|^2\)\tilde u_j^\eps-
\(K_{j,{\rm off}}^\eps\ast
  |v^\eps_{j+1}|^2\) v^\eps_j\\
&+\frac{1}{\eps}\(2\RE
  \tilde W_j^\eps\)\tilde u_j^\eps -\tilde\rho_j^\eps ,
    \end{aligned}
\right.
\end{equation*}
where the new source term is such that
\begin{equation*}
  \sup_{t\in [0,T]}\|\tilde \rho_j^\eps(t)\|_{\Sigma} \le C\eps.
\end{equation*}
Resuming the energy estimates used in Section~\ref{sec:bootstrap}, we
infer:
\begin{proposition}
Let $T>0$. Under the assumptions of Proposition~\ref{prop:reste2}, 
there exists $\eps_0>0$ such that for
$\eps\in(0,\eps_0]$, \eqref{eq:exactesimplif} has a unique
solution $(\tilde u_1^\eps,\tilde u_2^\eps)\in
C([0,T];\Sigma^4)^2$. Moreover, there exists $C$ independent of
$\eps\in (0,\eps_0]$ such that
\begin{equation*}
  \sup_{t\in [0,T]}\left\| \tilde u_1^\eps
    (t)-u_1(t)-\sqrt\eps u_1^{(1)}(t)\right\|_{\Sigma}
+\sup_{t\in [0,T]}\left\| \tilde u_2^\eps
    (t)-u_2(t)-\sqrt\eps u_2^{(1)}(t)\right\|_{\Sigma} \le C\eps. 
\end{equation*}  
\end{proposition}
Note that unlike in the proof of Theorem~\ref{theo:final}, no
bootstrap argument is needed at this stage, since we already have
uniform estimates for $\tilde u_j^\eps,u_j, u_j^{(1)}$ in
$C([0,T];\Sigma^4)$. We readily infer:
\begin{equation*}
  \theta_j^\eps(t)=\theta_j(t)+\O\(\sqrt\eps\)\quad \text{in }L^\infty([0,T]),
\end{equation*}
where $\theta_j$ is given by
\begin{equation}
  \label{eq:theta}
  \begin{aligned}
  \theta_j(t) &=  \int_0^t \nabla K(0)\cdot \(2\RE\int
z\overline u_j(s,z) u_j^{(1)}(s,z)dz\)ds\\
&+\int_0^t\nabla
K(q_j(s)-q_{j+1}(s))\cdot \(2\RE\int 
z\overline u_{j+1}(s,z) u_{j+1}^{(1)}(s,z)dz\)ds. 
\end{aligned}
\end{equation}
We have obviously $\theta_j\in C^1([0,T])$, and $\theta_j(0)=\dot
\theta_j(0)=0$. To see that $\theta_j\in C^2([0,T])$, in view of
the Cauchy--Schwarz inequality, and since $u_j,u_j^{(1)}\in
C([0,T];\Sigma)$, it suffices that
verify that $u_j,u_j^{(1)}\in C^1([0,T];L^2)$. This property is 
 a straightforward
consequence of Equations~\eqref{eq:systenv0} and \eqref{eq:reste2}, in view of
the regularity of $u_j$ and $u_j^{(1)}$. This completes the proof of
Theorem~\ref{theo:alpha0}. 
\smallbreak

To conclude, we check that the phase shifts $\theta_j$ are non-trivial
in general, by computing their initial second order derivatives: since
$u_{j\mid t=0}^{(1)}=0$, 
\begin{align*}
  \ddot \theta_j(0)&= \nabla K(0)\cdot \(2\RE\int
z\overline a_j(z) \d_t u_j^{(1)}(0,z)dz\)\\
&+\nabla
K(q_{j}(0)-q_{j+1}(0))\cdot \(2\RE\int 
z\overline a_{j+1}(z) \d_t u_{j+1}^{(1)}(0,z)dz\). 
\end{align*}
From \eqref{eq:reste2}, we have
\begin{equation*}
  i\d_t u_{j}^{(1)}(0,y)= \({\mathcal V}_j(0,y)+{\mathcal
    K}_{j,{\rm diag}}\ast|a_j|^2+ {\mathcal K}_{j,{\rm
      off}}\ast|a_{j+1}|^2\) a_j(y) , 
\end{equation*}
so the first line in the expression for $\ddot \theta_j(0)$ is zero, and
\begin{align*}
  \ddot \theta_j(0)&= \nabla
K(q_{j}(0)-q_{j+1}(0))\cdot \(\int 
z{\mathcal V}_j(0,z)2\IM\(\overline a_{j+1}a_j\)(z)dz\)\\
&+ \nabla
K(q_{j}(0)-q_{j+1}(0))\cdot \(\int \({\mathcal K}_{j,{\rm diag}}\ast|a_j|^2\)
2\IM\(\overline a_{j+1}a_j\)(z)dz\)\\
&+ \nabla
K(q_{j}(0)-q_{j+1}(0))\cdot \(\int \({\mathcal K}_{j,{\rm off}}\ast|a_{j+1}|^2\)
2\IM\(\overline a_{j+1}a_j\)(z)dz\).  
\end{align*}
Therefore in general $\theta_j\not\equiv 0$. 
\begin{remark}[Instability]
  The fact that it is necessary to analyze an $\O(\sqrt\eps)$
  correction to $(u_1,u_2)$ to compute $(\theta_1,\theta_2)$  implies
  the existence of 
  instabilities at the semi-classical level. Typically, a perturbation
  of the initial data at order $\eps^\gamma$ with $0<\gamma<1/2$ will
  affect the leading order behavior of $u^\eps$ in $L^2$ (for the strong
  topology) for some
  time $0<t^\eps \to 0$. On the other hand, since the $\theta_j$'s are
  purely time dependent, the Wigner measure are not affected by this
  phenomenon. Since the approach to describe this instability is the
  same as in \cite{CaARMA}, we simply refer to that paper for more
  details. 
\end{remark}

\subsubsection*{Acknowledgement} I am grateful to Alexander Mielke for
pointing out Remark~\ref{rem:hamil}.

\providecommand{\bysame}{\leavevmode\hbox to3em{\hrulefill}\thinspace}
\providecommand{\MR}{\relax\ifhmode\unskip\space\fi MR }
\providecommand{\MRhref}[2]{%
  \href{http://www.ams.org/mathscinet-getitem?mr=#1}{#2}
}
\providecommand{\href}[2]{#2}

 \end{document}